\documentclass[11pt]{amsart}
\usepackage[top=1.5in, bottom=1.5in, left=1.5in, right=1.5in]{geometry}
\allowdisplaybreaks
\numberwithin{equation}{section}

\usepackage{amsmath}
\usepackage{amssymb}
\usepackage{hyperref}
\usepackage{amsthm}
\usepackage{mathtools}
\usepackage[capitalize,noabbrev]{cleveref}
\usepackage{xcolor}
\usepackage{tikz}

\newtheorem{thm}{Theorem}[section]
\newtheorem{lemma}[thm]{Lemma}
\newtheorem{cor}[thm]{Corollary}

\newtheorem{conj}[thm]{Conjecture}

\newtheorem{Definition}[thm]{Definition}

\newtheorem{Example}[thm]{Example}
\newenvironment{example}
  {\begin{Example}\rm}{\end{Example}}

\newtheorem{Remark}[thm]{Remark}
\newenvironment{remark}
  {\begin{Remark}\rm}{\end{Remark}}
  
  \newtheorem{Question}[thm]{Question}
\newenvironment{question}
  {\begin{Question}\rm}{\end{Question}}
  
    \newtheorem{Problem}[thm]{Problem}

\crefname{thm}{Theorem}{Theorems}
\crefname{lemma}{Lemma}{Lemmas}
\crefname{cor}{Corollary}{Corollaries}
\crefname{prop}{Proposition}{Propositions}
\crefname{conj}{Conjecture}{Conjectures}

\crefname{definition}{Definition}{Definitions}
\crefname{example}{Example}{Examples}
\crefname{remark}{Remark}{Remarks}
\crefname{question}{Question}{Questions}
\crefname{problem}{Problem}{Problems}

\newcommand{\emailhref}[1]{\email{\href{#1}{#1}}}
\newcommand{\dfn}[1]{\textcolor{blue}{\emph{#1}}}

\title[Upho lattices II]{Upho lattices II: ways of realizing a core}
\author{Sam Hopkins}\emailhref{samuelfhopkins@gmail.com}
\address{Department of Mathematics, Howard University, Washington, DC}
\author{Joel B. Lewis}\emailhref{jblewis@gwu.edu}
\address{Department of Mathematics, George Washington University, Washington, DC}

\begin{document}

\begin{abstract}
A poset is called upper homogeneous, or ``upho,'' if all of its principal order filters are isomorphic to the whole poset. In previous work of the first author, it was shown that each (finite-type $\mathbb{N}$-graded) upho lattice has associated to it a finite graded lattice, called its core, which determines the rank generating function of the upho lattice. In that prior work the question of which finite graded lattices arise as cores was explored. Here, we study the question of in how many different ways a given finite graded lattice can be realized as the core of an upho lattice. We show that if the finite lattice has no nontrivial automorphisms, then it is the core of finitely many upho lattices. We also show that the number of ways a finite lattice can be realized as a core is unbounded, even when restricting to rank-two lattices. We end with a discussion of a potential algorithm for listing all the ways to realize a given finite lattice as a core.
\end{abstract}

\maketitle

\section{Introduction} \label{sec:intro}

A poset $\mathcal{P}$ is called \dfn{upper homogeneous}, or ``\dfn{upho},'' if for every $p \in \mathcal{P}$, the principal order filter $V_p = \{q\in \mathcal{P}\colon q \geq p\}$ is isomorphic to the original poset $\mathcal{P}$. This class of infinite, self-similar posets was introduced a few years ago by Stanley~\cite{stanley2020upho, stanley2021rational}, and has subsequently been shown by a number of authors~\cite{gao2020upho, hopkins2022note, hopkins2024upho1, fu2024upho} to be quite rich and fascinating.

In~\cite{hopkins2022note, hopkins2024upho1}, the first author explained how each (finite-type $\mathbb{N}$-graded) upho \emph{lattice} has associated to it a finite graded lattice that controls many features of the upho lattice. More precisely, for an upho lattice~$\mathcal{L}$, its \dfn{core} is $L \coloneqq [\hat{0},s_1\vee \cdots \vee s_r]$, the interval from the minimum $\hat{0}$ to the join of its atoms $s_1,\ldots,s_r$. In~\cite{hopkins2022note}, it was shown that $F(\mathcal{L};x) = \chi^*(L;x)^{-1}$, where $F(\mathcal{L};x)$ is the rank generating function of the upho lattice $\mathcal{L}$ and $\chi^*(L;x)$ is the (reciprocal) characteristic polynomial of its core~$L$. So the core determines how quickly the upho lattice grows.

In~\cite{hopkins2024upho1}, the question of which finite graded lattices are cores of upho lattices was explored. It was shown that this is a very subtle question: many important finite lattices are cores, but many also are not. In general, it is unclear how one can determine if a given finite lattice is the core of some upho lattice.

Importantly, the core does not determine the upho lattice, in the sense that the same finite lattice can be the core of multiple upho lattices. For example, as depicted in \cref{fig:b2_ex}, the rank-two Boolean lattice $B_2$ is the core of two different upho lattices.\footnote{We will show in \cref{sec:rank_two} that these are the \emph{only} two upho lattices which have $B_2$ as their core.} In this paper, we study the question of in how many different ways a given finite lattice can be realized as a core. In a sense, we study whether every upho lattice can be represented by a finite amount of data.

\begin{figure}
    \begin{center}
    \begin{tikzpicture}[scale=0.55]
    \node[circle, inner sep=1pt, fill=black] (1) at (0,0) {};
    \node[circle, inner sep=1pt, fill=black] (2) at (-1,1) {};
    \node[circle, inner sep=1pt, fill=black] (3) at (1,1) {};
    \node[circle, inner sep=1pt, fill=black] (4) at (-2,2) {};
    \node[circle, inner sep=1pt, fill=black] (5) at (0,2) {};
    \node[circle, inner sep=1pt, fill=black] (6) at (2,2) {};
    \node[circle, inner sep=1pt, fill=black] (7) at (-3,3) {};
    \node[circle, inner sep=1pt, fill=black] (8) at (-1,3) {};
    \node[circle, inner sep=1pt, fill=black] (9) at (1,3) {};
    \node[circle, inner sep=1pt, fill=black] (10) at (3,3) {};
    \node[circle, inner sep=1pt, fill=black] (11) at (-4,4) {};
    \node[circle, inner sep=1pt, fill=black] (12) at (-2,4) {};
    \node[circle, inner sep=1pt, fill=black] (13) at (0,4) {};
    \node[circle, inner sep=1pt, fill=black] (14) at (2,4) {};
    \node[circle, inner sep=1pt, fill=black] (15) at (4,4) {};
    \node (16) at (-5,5) {};
    \node (17) at (-3,5) {};
    \node (18) at (-1,5) {};
    \node (19) at (1,5) {};
    \node (20) at (3,5) {};
    \node (21) at (5,5) {};
    \draw[thick] (1) -- (2) -- (4) -- (7) -- (11);
    \draw[thick] (3) -- (5) -- (8) -- (12);
    \draw[thick] (6) -- (9) -- (13);
    \draw[thick] (10) -- (14);
    \draw[thick] (1) -- (3) -- (6) -- (10) -- (15);
    \draw[thick] (2) -- (5) -- (9) -- (14);
    \draw[thick] (4) -- (8) -- (13);
    \draw[thick] (7) -- (12);
    \draw (16) -- (11) -- (17) -- (12) -- (18) -- (13) -- (19) -- (14) -- (20) -- (15) -- (21);
    \end{tikzpicture} 
    \qquad \vrule \qquad 
    \begin{tikzpicture}[scale=0.55]
    \node[circle, inner sep=1pt, fill=black] (1) at (0,0) {};
    \node[circle, inner sep=1pt, fill=black] (2) at (-1,1) {};
    \node[circle, inner sep=1pt, fill=black] (3) at (1,1) {};
    \node[circle, inner sep=1pt, fill=black] (4) at (-2,2) {};
    \node[circle, inner sep=1pt, fill=black] (5) at (0,2) {};
    \node[circle, inner sep=1pt, fill=black] (6) at (2,2) {};
    \node[circle, inner sep=1pt, fill=black] (7) at (-3,3) {};
    \node[circle, inner sep=1pt, fill=black] (8) at (-1,3) {};
    \node[circle, inner sep=1pt, fill=black] (9) at (1,3) {};
    \node[circle, inner sep=1pt, fill=black] (10) at (3,3) {};
    \node[circle, inner sep=1pt, fill=black] (11) at (-4,4) {};
    \node[circle, inner sep=1pt, fill=black] (12) at (-2,4) {};
    \node[circle, inner sep=1pt, fill=black] (13) at (0,4) {};
    \node[circle, inner sep=1pt, fill=black] (14) at (2,4) {};
    \node[circle, inner sep=1pt, fill=black] (15) at (4,4) {};
    \node (16) at (-5,5) {};
    \node (17) at (-3,5) {};
    \node (18) at (-1,5) {};
    \node (19) at (1,5) {};
    \node (20) at (3,5) {};
    \node (21) at (5,5) {};
    \draw[thick] (1) -- (2) -- (5) -- (9) -- (13);
    \draw[thick] (1) -- (3) -- (6) -- (10) -- (15);
    \draw[thick] (3) -- (5);
    \draw[thick] (2) -- (4) -- (7) -- (11);
    \draw[thick] (4) -- (9) -- (6);
    \draw[thick] (5) -- (8) -- (12);
    \draw[thick] (7) -- (13) -- (8) -- (13) -- (10);
    \draw[thick] (9) -- (14);
    \draw (11) -- (18) -- (12) -- (18) -- (13) -- (18) -- (14) -- (18) -- (15) -- (18);
    \draw (11) -- (16);
    \draw (12) -- (17);
    \draw (13) -- (19);
    \draw (14) -- (20);
    \draw (15) -- (21);
    \end{tikzpicture}
    \end{center}
    \caption{Two different upho lattices with core $B_2$.}
    \label{fig:b2_ex}
\end{figure}
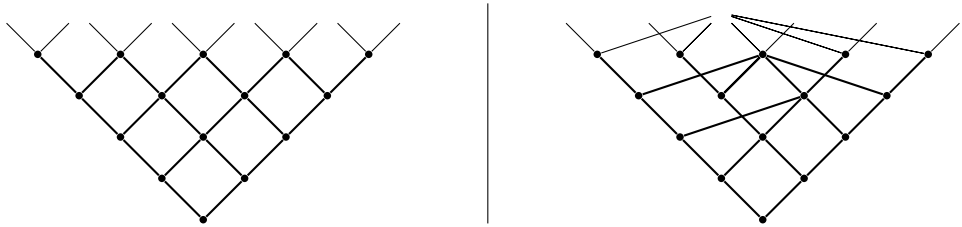

Specifically, for a finite graded lattice $L$, let $\kappa(L)$ be the number of different upho lattices of which $L$ is a core. We are interested in the behavior of the function $\kappa(L)$. We will see that the way this function behaves is also quite subtle: in some ways, this function is ``small;'' in other ways, it is ``big.''

Concerning the ``smallness'' of $\kappa(L)$, our first major result, proved in \cref{sec:color}, says that if the finite lattice~$L$ has no nontrivial automorphisms, then $\kappa(L)$ is finite. This result suggests that~$\kappa(L)$ may always be finite for all finite lattices $L$. But we currently cannot rule out the possibility that $\kappa(L)$ is infinite, even uncountably infinite, for some finite lattice $L$ with a nontrivial automorphism.

Concerning the ``bigness'' of $\kappa(L)$, our second major result, established in \cref{sec:rank_two}, says that this function is unbounded. Actually, as soon as we know there is one lattice $L$ with $\kappa(L) > 1$, a simple product construction implies that $\kappa(L)$ is unbounded. But we show, what is much less trivial, that $\kappa(L)$ is unbounded even when restricting to lattices $L$ of rank two. There is a unique rank-two graded lattice with $n$ atoms, denoted $M_n$, and we show more precisely that for each $n \geq 2$, $\kappa(M_n) \geq p(n)$, where $p(n)$ is the number of integer partitions of $n$.

In the proof of both of our major results, monoids are used in an essential way. Previous work~\cite{gao2020upho, hopkins2024upho1, fu2024upho} highlighted the close connection between monoids and upho posets. The key observation is that each (homogeneously finitely generated) left-cancellative monoid gives rise to a (finite-type $\mathbb{N}$-graded) upho poset. 

An intriguing open question is whether \emph{every} upho lattice comes from a monoid in this way.\footnote{In fact, as far as we know, it is possible that every (finite-type $\mathbb{N}$-graded) upho poset comes from a (homogeneously finitely generated) left-cancellative monoid. \cite[Remark~5.1]{gao2020upho} says ``[the monoid construction] is not a method to construct \emph{all} upho posets,'' but Yibo Gao has told us that he does not know an example of an upho poset which definitively does not come from a monoid.} If so, it would mean that this order-theoretic definition has intrinsic algebraic content. It would also imply the finiteness of $\kappa(L)$ for all finite lattices $L$, and might lead to an algorithm for listing all upho lattices of which $L$ is a core. We discuss some speculations along these lines in the final section, \cref{sec:algorithm}.

\subsection*{Acknowledgments} 

We thank Ziyao Fu, Yibo Gao, Jon McCammond, Yulin Peng, David Speyer, and Yuchong Zhang for useful comments related to this work. SageMath~\cite{Sage} was an important computational aid in this research. Finally, we thank the anonymous referees for their careful reading of our paper, and useful comments.

\section{Background on upho posets and monoids} \label{sec:background}

\subsection{Posets} 

We follow standard terminology for posets as laid out for instance in~\cite[Chapter~3]{stanley2012ec1}. We also follow the notation in~\cite[\S2]{hopkins2024upho1}, where these preliminaries on posets are reviewed in more detail.

Let~$P=(P,\leq)$ be a poset. We use $\hat{0}$ and $\hat{1}$ to denote the \dfn{minimum} and \dfn{maximum} of $P$ when they exist. We use $\lessdot$ for the \dfn{cover} relation of $P$. An \dfn{atom} of $P$ is an element $s\in P$ with $\hat{0}\lessdot s$.

A poset $L$ is a \dfn{meet semilattice} if every pair of elements $x,y\in L$ has a \dfn{meet}, i.e., greatest lower bound, denoted $x\wedge y$, and it is a \dfn{join semilattice} if every pair of elements $x,y\in L$ has a \dfn{join}, i.e., least upper bound, denoted $x \vee y$. The poset $L$ is a \dfn{lattice} if it is both a meet and join semilattice.

Some posets $P$ we work with will be infinite. But every such $P$ will be at least \dfn{locally finite}, which means that all \dfn{intervals} $[x,y] \coloneqq \{z \in P\colon x \leq z \leq y\}$ in~$P$ are finite. We recall that the \dfn{M\"{o}bius function} $\mu(x,y)$ of a locally finite poset $P$ can be defined recursively by $\mu(x,x) \coloneqq 1$ for all $x\in P$ and $\mu(x,y) \coloneqq -\sum_{x\leq z < y} \mu(x,z)$ for~$x < y \in P$. 

Since we routinely work with both finite and infinite posets, so from now on we use the convention that normal script letters (like $P$ and $L$) denote finite posets while calligraphic letters (like $\mathcal{P}$ and $\mathcal{L}$) denote infinite posets.

A finite poset $P$ is \dfn{graded} (of \dfn{rank} $n$) if it has a minimum $\hat{0}$, a maximum $\hat{1}$, and it can be written as a disjoint union $P=P_0\sqcup P_1 \sqcup \cdots \sqcup P_n$ such that every maximal chain in $P$ is of the form $\hat{0}=x_0 \lessdot x_1 \lessdot \cdots \lessdot x_n=\hat{1}$ with~$x_i \in P_i$. In this case, the \dfn{rank function} $\rho \colon P \to \{0,1,\ldots,n\}$ of $P$ is given by $\rho(x) = i$ if $x \in P_i$. 
The \dfn{rank generating function} of $P$ is then defined to be
\[    F(P;x) \coloneqq \sum_{p\in P} x^{\rho(p)},\]
and its \dfn{(reciprocal) characteristic polynomial} is defined to be
\[    \chi^*(P;x) \coloneqq  \sum_{p\in P} \mu(\hat{0},p) \, x^{\rho(p)} \, .\]
For example, the \dfn{Boolean lattice}~$B_n$ of subsets of $[n] \coloneqq  \{1,\ldots,n\}$ ordered by inclusion is graded of rank $n$, and we have $F(B_n;x)=(1+x)^n$ and $\chi^*(B_n;x)=(1-x)^n$.

We use $\mathbb{N}\coloneqq  \{0,1,2,\ldots\}$ for the natural numbers. An infinite poset $\mathcal{P}$ is \dfn{$\mathbb{N}$-graded} if it has a minimum $\hat{0}$ and it can be written as a disjoint union $\mathcal{P} = \bigsqcup_{i=0}^{\infty} \mathcal{P}_i$ such that every maximal chain in $\mathcal{P}$ is of the form $\hat{0}=x_0 \lessdot x_1 \lessdot x_2 \lessdot \cdots$ with $x_i \in \mathcal{P}_i$. In this case, the poset's \dfn{rank function} $\rho \colon \mathcal{P} \to \mathbb{N}$ is given by $\rho(x) = i$ if $x \in \mathcal{P}_i$. We say such a $\mathcal{P}$ is \dfn{finite-type} $\mathbb{N}$-graded if~$\#\mathcal{P}_i < \infty$ for all $i$. The \dfn{rank generating} and \dfn{characteristic generating functions} of such a $\mathcal{P}$ are then defined to be
\[
    F(\mathcal{P};x) \coloneqq  \sum_{p\in \mathcal{P}} x^{\rho(p)} \qquad \text{ and } \qquad
    \chi^*(\mathcal{P};x) \coloneqq  \sum_{p\in \mathcal{P}} \mu(\hat{0},p) \, x^{\rho(p)} \, ,
\]
respectively.  For example, $\mathbb{N}^n$, with the usual Cartesian product partial order, is finite-type $\mathbb{N}$-graded with $F(\mathbb{N}^n;x)=\frac{1}{(1-x)^n}$ and $\chi^*(\mathbb{N}^n;x)=(1-x)^n$.

\subsection{Upho posets}  

A poset $\mathcal{P}$ is \dfn{upper homogeneous (upho)} if for every~$p\in \mathcal{P}$, the corresponding \dfn{principal order filter} $V_p \coloneqq  \{q\in \mathcal{P}\colon q \geq p\}$ is isomorphic to~$\mathcal{P}$. To avoid trivialities, we assume all upho posets have at least two elements; then, they must be infinite. In order to be able to apply the tools of enumerative and algebraic combinatorics to study them, {\bf all upho posets are assumed finite-type $\mathbb{N}$-graded} from now on. For example, $\mathbb{N}^n$ is an upho lattice.

\begin{remark}
In \cite{fu2024upho}, the authors consider other, weaker finiteness conditions for upho posets. We believe that much of what we do in this paper could be adapted to those more general classes of upho posets, although certainly \emph{some} finiteness condition is needed. But for simplicity, and because the class of finite-type $\mathbb{N}$-graded upho posets is already very rich, we restrict our attention to this class.
\end{remark}

As observed previously by the first author, upho posets have an interesting symmetry regarding their rank and characteristic generating functions.

\begin{thm}[{\cite[Theorem~1]{hopkins2022note}}] \label{thm:upho_gfs}
If $\mathcal{P}$ is an upho poset, then $F(\mathcal{P};x) = \chi^*(\mathcal{P};x)^{-1}$.
\end{thm}

For upho lattices, we can say more. For an upho lattice~$\mathcal{L}$, we define its \dfn{core} to be $L \coloneqq  [\hat{0},s_1 \vee \cdots \vee s_r]$, the interval from its minimum $\hat{0}$ to the join of its atoms~$s_1,\ldots,s_r$. Notice that the core of an upho lattice is a finite graded lattice. An easy corollary of \cref{thm:upho_gfs} is the following.

\begin{cor}[{\cite[Corollary~6]{hopkins2022note}}] \label{cor:upho_gfs}
If $\mathcal{L}$ is an upho lattice, then $F(\mathcal{L};x) = \chi^*(L;x)^{-1}$, where $L$ is the core of $\mathcal{L}$.
\end{cor}

\Cref{cor:upho_gfs} says the core of an upho lattice determines how quickly it grows. For instance, $\mathbb{N}^n$ is an upho lattice with core $B_n$, and $F(\mathbb{N}^n;x)=\frac{1}{(1-x)^n}=\chi^*(B_n;x)^{-1}$.
However, the core does \emph{not} completely determine the upho lattice, in the sense that a given finite graded lattice can be the core of multiple different upho lattices. For example, there are two different upho lattices with core $B_2$, depicted in \cref{fig:b2_ex}.

For a finite graded lattice $L$, we let $\kappa(L)$ denote the cardinality of the set of different upho lattices with core $L$. Our main interest in this paper is in understanding the function $\kappa(L)$. A priori $\kappa(L)$ could be infinite, even uncountably infinite, for some $L$. But we will provide some reasons to think $\kappa(L)$ might be finite for all $L$.

\subsection{Monoids}

An important source of upho posets are monoids, as has been observed previously in~\cite{gao2020upho, hopkins2024upho1, fu2024upho}. Let us review the connection. For basics on monoids, consult, e.g.,~\cite{dehornoy1999gaussian, dehornoy2015foundations}. We also follow the notation of \cite[\S4]{hopkins2024upho1}, where again these preliminaries on monoids are reviewed in more detail.

Let $M=(M,\cdot)$ be a monoid. We say that~$M$ is \dfn{left-cancellative} if $ab=ac$ implies that~$b=c$ for $a,b,c\in M$. Left-cancellativity is a version of upper homogeneity. But we also need to enforce our finiteness requirements. Let us say that $M$ is \dfn{homogeneously finitely generated} if it has a presentation $M=\langle S \mid R \rangle$ where the set $S$ of generators is finite and where every relation in $R$ is homogeneous, i.e., of the form $w=w'$ with $\ell(w) = \ell(w')$, where~$\ell(w)$ denotes the length of the word~$w$. Finally, we use $\leq_L$ to denote the \dfn{left divisibility} relation on $M$: $a \leq_L b$ for~$a,b \in M$ means that~$b=ac$ for some $c\in M$. 

The next lemma summarizes the connection between monoids and upho posets.

\begin{lemma}[{\cite[Lemma~5.1]{gao2020upho}; see also~\cite{hopkins2024upho1, fu2024upho}}] \label{lem:monoids}
Let $M$ be a left-cancellative, homogeneously finitely generated monoid. Then $(M,\leq_L)$ is a (finite-type $\mathbb{N}$-graded) upho poset.
\end{lemma}

For example, the free commutative monoid $\langle s_1,\ldots,s_n \mid s_is_j = s_js_i \rangle$ satisfies the conditions of \cref{lem:monoids}, and gives us the upho lattice $\mathbb{N}^n$.

\section{Colorable upho posets, automorphisms, and finiteness} \label{sec:color}

In this section we focus on proving that $\kappa(L)$ is finite. Our major result says that~$\kappa(L)$ is finite when the finite lattice $L$ has no nontrivial automorphisms. We use the connection to monoids to prove this result.

It is helpful to recast the monoid construction of upho posets in a slightly more combinatorial framework (cf.~\cite[\S3]{fu2024upho}). An \dfn{upho coloring} of a finite-type $\mathbb{N}$-graded poset~$\mathcal{P}$ is a function $c$ mapping each cover relation of $\mathcal{P}$ to an atom of $\mathcal{P}$, such that:
\begin{itemize}
\item $c(\hat{0}\lessdot s) = s$ for every atom $s \in \mathcal{P}$, and
\item for each $p \in \mathcal{P}$, there is an isomorphism $\varphi_p\colon V_p \to \mathcal{P}$ which is \dfn{color-preserving} in the sense that $c(\varphi_p(x) \lessdot \varphi_p(y))=c(x \lessdot y)$ for all $x \lessdot y\in V_p$.
\end{itemize}
Of course, if $\mathcal{P}$ has an upho coloring $c$, then it is an upho poset. Let us call a $\mathcal{P}$ together with such a $c$ a \dfn{colored upho poset}, and call an upho poset $\mathcal{P}$ \dfn{colorable} if it admits such a $c$.

The colorable upho posets are exactly those coming from monoids.

\begin{lemma}[{\cite[Corollary~3.6]{fu2024upho}}] \label{lem:monoids_colors}
There is a bijective correspondence between left-cancellative, homogeneously finitely generated monoids $M$ and colored upho posets~$\mathcal{P}$. Given such a monoid $M$, we let~$\mathcal{P} \coloneqq (M,\leq_L)$ as in \cref{lem:monoids}, with the coloring given by $c(x\lessdot y) \coloneqq s$ if $y=xs$. Conversely, given such a colored upho poset $\mathcal{P}$, the associated monoid $M$ has presentation
\[ M \coloneqq \left\langle s_1,\ldots,s_r\mid \parbox{3.25in}{\begin{center} $c(\hat{0}=x_0\lessdot x_1)c(x_1\lessdot x_2)\cdots c(x_{k-1}\lessdot x_k = p) =$ \\ $c(\hat{0}=y_0\lessdot y_1)c(y_1\lessdot y_2)\cdots c(y_{k-1}\lessdot y_k = p)$ \end{center} }\right\rangle\]
where the generators are the atoms $s_1,\ldots,s_r$ of $\mathcal{P}$, and the relations correspond to all pairs of saturated chains $\hat{0}=x_0\lessdot x_1\lessdot \cdots \lessdot x_k=p$, $\hat{0}=y_0\lessdot y_1\lessdot \cdots \lessdot y_k=p$ from $\hat{0}$ to any $p \in \mathcal{P}$.
\end{lemma}

\begin{proof}
As indicated, this is essentially proved in \cite[Corollary~3.6]{fu2024upho}. However, rather than define the monoid $M$ corresponding to a colored upho poset $\mathcal{P}$ by generators and relations, the authors there define $M$ in a slightly different way, as we now explain. First, they observe that, for a fixed upho coloring $c$ of $\mathcal{P}$, the color-preserving isomorphisms $\varphi_p\colon V_p \to \mathcal{P}$ for $p \in \mathcal{P}$ are uniquely determined. Then, they define the monoid~$M$ to have as its set of elements the elements of $\mathcal{P}$, with product given by $pq = \varphi_p^{-1}(q)$ for $p,q \in \mathcal{P}$. But we can see that their definition of $M$ is equivalent to the one we gave above in terms of generators and relations by identifying each element $p \in \mathcal{P}$ with the product $c(x_0\lessdot x_1)c(x_1\lessdot x_2)\cdots c(x_{k-1}\lessdot x_k)$ for any saturated chain $\hat{0}=x_0\lessdot x_1\lessdot \cdots \lessdot x_k=p$ from $\hat{0}$ to $p$. The relations mean that the choice of saturated chain does not matter, because all possible choices are identified with each other.
\end{proof}

\begin{remark} \label{rem:monoid_corresp}
From now on, we will frequently use the correspondence between the elements of a colored upho poset $\mathcal{P}$ and the elements of its associated monoid~$M$: the element $p \in \mathcal{P}$ in the poset corresponds to the element $c(x_0\lessdot x_1)\cdots c(x_{k-1}\lessdot x_k) \in M$ in the monoid, where $\hat{0}=x_0\lessdot x_1 \lessdot \cdots \lessdot x_k = p$ is any maximal chain from $\hat{0}$ to $p$.
\end{remark}

In the case of lattices, we can substantially reduce the set of relations we need.

\begin{lemma} \label{lem:monoids_colors_lat}
Let $\mathcal{L}$ be a colored upho lattice. Then the monoid $M$ associated to it by \cref{lem:monoids_colors} is
\begin{gather} \label{eq:monoid}
M = \left\langle s_1,\ldots,s_r\mid \parbox{3.6in}{\begin{center} $c(\hat{0}=x_0\lessdot x_1)c(x_1\lessdot x_2)\cdots c(x_{k-1}\lessdot x_k = x_1 \vee y_1) =$ \\ $c(\hat{0}=y_0\lessdot y_1)c(y_1\lessdot y_2)\cdots c(y_{k-1}\lessdot y_k = x_1 \vee y_1)$ \end{center} }\right\rangle
\end{gather}
where the generators are the atoms $s_1,\ldots,s_r$ of $\mathcal{L}$, and the relations correspond to all pairs of saturated chains $\hat{0}=x_0\lessdot x_1\lessdot \cdots \lessdot x_k=x_1 \vee y_1$, $\hat{0}=y_0\lessdot y_1\lessdot \cdots \lessdot y_k=x_1 \vee y_1$ from $\hat{0}$ to the join $x_1 \vee y_1$ of two atoms $x_1,y_1$ of $\mathcal{L}$.
\end{lemma}

\begin{proof}

Let us use $M$ to denote the monoid associated to our colored upho lattice~$\mathcal{L}$ by \cref{lem:monoids_colors}, and let us use $M'$ to denote the monoid with presentation the right-hand side of~\eqref{eq:monoid}. Our goal is to show that $M$ and $M'$ are the same. 

Let $S = \{s_1,\ldots,s_r\}$ be the set of atoms of $\mathcal{L}$, and let $S^*$ denote the set of words over the alphabet $S$. For two words $w_1,w_2 \in S^*$, we write $w_1 = w_2$ if~$w_1$ and $w_2$ are equal when viewed as elements in $M$, and similarly write $w_1 =' w_2$ for equality in $M'$. We want to show that for any $w_1,w_2 \in S^*$, we have $w_1 = w_2$ if and only if~$w_1 =' w_2$. Since the defining relations of $M'$ are a subset of the defining relations of $M$, if $w_1 =' w_2$ then clearly $w_1 = w_2$.

So now suppose $w_1,w_2 \in S^*$ satisfy $w_1 = w_2$ in $M$. Because the relations defining~$M$ are homogeneous, $w_1$ and $w_2$ have the same length $n \geq 0$. We prove that $w_1 =' w_2$ by induction on $n$. If $n = 0$ the claim is trivial, so suppose $n > 0$. If for some $s \in S$ we have $w_1 = su_1$ and $w_2 = su_2$, then $u_1 = u_2$ in $M$ because $M$ is left-cancellative; then by induction $u_1 =' u_2$, and so $w_1 =' w_2$.  Otherwise, suppose without loss of generality that the word $w_1$ starts with the atom $s_1$, and $w_2$ starts with $s_2$. Let $x \coloneqq s_1 \vee s_2\in \mathcal{L}$.  Let $u_1, u_2 \in S^*$ be words which when viewed as elements of $M$ are identified with $x \in \mathcal{L}$ (as in Remark~\ref{rem:monoid_corresp}), and are such that~$u_1$ starts with $s_1$ while~$u_2$ starts with $s_2$. Certainly $u_1=u_2$. But also, because the relations in $M'$ equate words corresponding to saturated chains from two atoms to their join, we in fact have $u_1 =' u_2$. Now let $y \in \mathcal{L}$ be the element identified with $w_1=w_2\in M$. Since~$y$ is an upper bound for $s_1$ and $s_2$, and $x$ is their least upper bound, we have that~$x \leq y$. Thus, we can find a word $v \in S^*$ so that $u_1v = u_2v \in M$ is identified with $y \in \mathcal{L}$. Then $w_1 = u_1v$, and both of these words start with $s_1$. Since $M$ is left-cancellative, we get an equality $\overline{w}_1 = \overline{u}_1v$, where $\overline{w}_1$ is the result of removing the initial $s_1$ from~$w_1$ and $\overline{u}_1$ is the result of removing the initial $s_1$ from $u_1$. But notice that the equality~$\overline{w}_1 = \overline{u}_1v$ involves words of length $n-1$, so by induction we have that~$\overline{w}_1 =' \overline{u}_1v$, which then implies $w_1='u_1v$. A symmetric argument shows that~$w_2='u_2v$. Finally, recalling $u_1='u_2$, we conclude $w_1='w_2$, as desired.
\end{proof}

From~\cref{lem:monoids_colors_lat} we easily conclude the following.

\begin{cor} \label{cor:monoid_lat_finiteness}
Each finite graded lattice is the core of finitely many colorable upho lattices.
\end{cor}

\begin{proof}
\Cref{lem:monoids_colors_lat} implies that if $\mathcal{L}$ is a colored upho lattice with core~$L$, then the way all of $\mathcal{L}$ is colored is determined by the way its core $L$ is colored. And of course there are only finitely many ways to color the finite poset $L$.
\end{proof}

But how to construct upho colorings? Let $\mathcal{P}$ be an upho poset. A \dfn{system of isomorphisms} for $\mathcal{P}$ is a collection of isomorphisms $\varphi_p\colon V_p \to \mathcal{P}$ for each $p \in \mathcal{P}$. (Such systems always exist, by the definition of upper homogeneity.) Let us say that a system of isomorphisms is \dfn{compatible} if
\begin{equation} \label{eq:compatibility}
\varphi_q = \varphi_{\varphi_p(q)} \circ (\varphi_p\big|_{V_q})
\end{equation}
for each $p \leq q \in \mathcal{P}$. Compatible systems of isomorphisms give us upho colorings.

\begin{lemma} \label{lem:compatible_coloring}
An upho poset is colorable if and only if it has a compatible system of isomorphisms.
\end{lemma}

\begin{proof}
Given a colored upho poset $\mathcal{P}$, the system of isomorphisms certifying that its coloring $c$ is upho must be compatible. Conversely, given a compatible system of isomorphisms $\varphi_p$ for $\mathcal{P}$, the coloring $c(x \lessdot y) \coloneqq  \varphi_{x}(y)$ is upho.
\end{proof}

\begin{figure}
    \begin{center}
    \begin{tikzpicture}[scale=0.75]
    \node[circle, inner sep=1pt, fill=black] (1) at (0,0) {};
    \node[circle, inner sep=1pt, fill=black, label=180:{\color{red}$s_1$}] (2) at (-1,1) {};
    \node[circle, inner sep=1pt, fill=black, label=0:{\color{blue}$s_2$}] (3) at (1,1) {};
    \node[circle, inner sep=1pt, fill=black] (4) at (-2,2) {};
    \node[circle, inner sep=1pt, fill=black] (5) at (0,2) {};
    \node[circle, inner sep=1pt, fill=black] (6) at (2,2) {};
    \node[circle, inner sep=1pt, fill=black] (7) at (-3,3) {};
    \node[circle, inner sep=1pt, fill=black] (8) at (-1,3) {};
    \node[circle, inner sep=1pt, fill=black] (9) at (1,3) {};
    \node[circle, inner sep=1pt, fill=black] (10) at (3,3) {};
    \node[circle, inner sep=1pt, fill=black] (11) at (-4,4) {};
    \node[circle, inner sep=1pt, fill=black] (12) at (-2,4) {};
    \node[circle, inner sep=1pt, fill=black] (13) at (0,4) {};
    \node[circle, inner sep=1pt, fill=black] (14) at (2,4) {};
    \node[circle, inner sep=1pt, fill=black] (15) at (4,4) {};
    \draw[red,thick] (1) -- (2) -- (4) -- (7) -- (11);
    \draw[red,thick] (3) -- (5) -- (8) -- (12);
    \draw[red,thick] (6) -- (9) -- (13);
    \draw[red,thick] (10) -- (14);
    \draw[blue,thick] (1) -- (3) -- (6) -- (10) -- (15);
    \draw[blue,thick] (2) -- (5) -- (9) -- (14);
    \draw[blue,thick] (4) -- (8) -- (13);
    \draw[blue,thick] (7) -- (12);
    \end{tikzpicture} 
    \qquad \vrule \qquad 
    \begin{tikzpicture}[scale=0.75]
    \node[circle, inner sep=1pt, fill=black] (1) at (0,0) {};
    \node[circle, inner sep=1pt, fill=black, label=180:{\color{red}$s_1$}] (2) at (-1,1) {};
    \node[circle, inner sep=1pt, fill=black, label=0:{\color{blue}$s_2$}] (3) at (1,1) {};
    \node[circle, inner sep=1pt, fill=black] (4) at (-2,2) {};
    \node[circle, inner sep=1pt, fill=black] (5) at (0,2) {};
    \node[circle, inner sep=1pt, fill=black] (6) at (2,2) {};
    \node[circle, inner sep=1pt, fill=black] (7) at (-3,3) {};
    \node[circle, inner sep=1pt, fill=black] (8) at (-1,3) {};
    \node[circle, inner sep=1pt, fill=black] (9) at (1,3) {};
    \node[circle, inner sep=1pt, fill=black] (10) at (3,3) {};
    \node[circle, inner sep=1pt, fill=black] (11) at (-4,4) {};
    \node[circle, inner sep=1pt, fill=black] (12) at (-2,4) {};
    \node[circle, inner sep=1pt, fill=black] (13) at (0,4) {};
    \node[circle, inner sep=1pt, fill=black] (14) at (2,4) {};
    \node[circle, inner sep=1pt, fill=black] (15) at (4,4) {};
    \draw[red,thick] (1) -- (2) -- (5) -- (9) -- (13);
    \draw[blue,thick] (1) -- (3) -- (6) -- (10) -- (15);
    \draw[red,thick] (3) -- (5);
    \draw[blue,thick] (2) -- (4) -- (7) -- (11);
    \draw[red,thick] (4) -- (9) -- (6);
    \draw[blue,thick] (5) -- (8) -- (12);
    \draw[red,thick] (7) -- (13) -- (8) -- (13) -- (10);
    \draw[blue,thick] (9) -- (14);
    \end{tikzpicture}
    \end{center}
    \caption{Two different colored upho lattices with core $B_2$, as in \cref{ex:bool_upho1} and \cref{ex:bool_upho2}. Also, compare \cref{fig:b2_ex}.} \label{fig:b2_ex_color}
\end{figure}
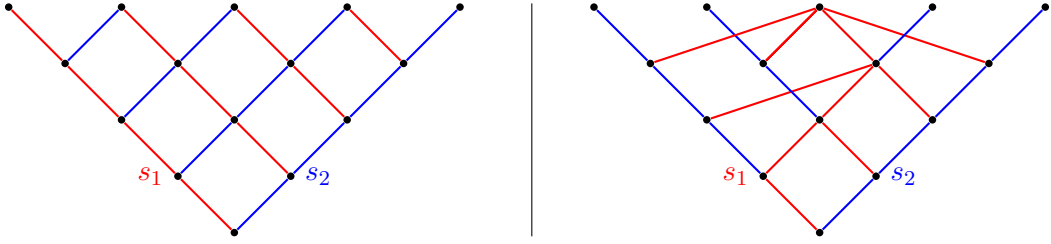

\begin{example} \label{ex:bool_upho1}
Let $n\geq 1$ and let $\mathcal{L} = \mathbb{N}^n$. Then $\mathcal{L}$ is an upho lattice, with core~$B_n$, and $\varphi_{(x_1,\ldots,x_n)}(y_1,\ldots,y_n) = (y_1-x_1,\ldots,y_n-x_n)$ gives a compatible system of isomorphisms for this $\mathcal{L}$. The corresponding monoid is the free commutative monoid as we have seen, i.e., $M = \langle s_1,\ldots,s_n\mid s_is_j = s_js_i \, \textrm{ for } \, 1\leq i < j \leq n\rangle$. For~$n=2$, the colored upho lattice $\mathcal{L}$ is depicted on the left in \cref{fig:b2_ex_color}. 
\end{example}

\begin{example} \label{ex:bool_upho2}
Let $n\geq 1$ and let $\mathcal{L} = \{\textrm{finite } \, S \subseteq \{1,2,\ldots\}\colon \max(S) < \#S + n\}$ (with the convention $\max(\varnothing)=0$), partially ordered by inclusion. Then $\mathcal{L}$ is another upho lattice with core~$B_n$: see~\cite[Remark~8]{hopkins2022note} and~\cite[\S3.5.1]{hopkins2024upho1}. A compatible system of isomorphisms for this $\mathcal{L}$ is~$\varphi_S(T) = f_S(T\setminus S)$ where $f_S \colon \{1,2,\ldots\}\setminus S \to \{1,2,\ldots\}$ is the unique order-preserving bijection. And the corresponding monoid is then $M = \langle s_1,\ldots,s_n\mid s_is_{j-1} = s_js_i \, \textrm{ for } \, 1\leq i < j \leq n\rangle$. Compare this monoid to the one in \cref{ex:bool_upho1}. For~$n=2$, the colored upho lattice~$\mathcal{L}$ is depicted on the right in \cref{fig:b2_ex_color}.
\end{example}

\begin{remark}
The upho lattices from~\cref{ex:bool_upho2} belong to a more general construction for any \emph{uniform sequence of supersolvable geometric lattices}: see~\cite[\S3]{hopkins2024upho1}. These sequences include not just Boolean lattices, but also subspace lattices, partition lattices, etc. The conditions which go into the definition of ``uniform sequence'' imply the resulting systems of isomorphisms are compatible. Therefore, all the upho lattices in~\cite[\S3]{hopkins2024upho1} come from monoids. Also, because of semimodularity, all the monoids produced in this way will in fact have quadratic defining relations.
\end{remark}

So when does an upho poset $\mathcal{P}$ has a \emph{compatible} system of isomorphisms? It turns out that $\mathcal{P}$ having no nontrivial automorphisms is enough to guarantee this.

\begin{cor} \label{cor:colorability}
If an upho poset has no nontrivial automorphisms, then it is colorable.
\end{cor}

\begin{proof}
Let $\mathcal{P}$ be an upho poset and let $\varphi_p$, for $p \in \mathcal{P}$, be a system of isomorphisms for~$\mathcal{P}$. Notice that if~\eqref{eq:compatibility} fails for some $p\leq q \in \mathcal{P}$, then the left-hand and right-hand sides are two different isomorphisms $V_q\to \mathcal{P}$. Composing one of these isomorphisms with the inverse of the other would then yield a \emph{nontrivial automorphism} $\mathcal{P}\to \mathcal{P}$. So if $\mathcal{P}$ has no nontrivial automorphisms, then in fact the system of isomorphisms is compatible, and hence by \cref{lem:compatible_coloring}, $\mathcal{P}$ is colorable.
\end{proof}

To finish the proof of our first main result, we need the following observation.

\begin{lemma} \label{lem:autos}
If a finite graded lattice $L$ has no nontrivial automorphisms, then any upho lattice with core $L$ also has no nontrivial automorphisms.
\end{lemma}

\begin{proof}
Let $L$ be a finite graded lattice that has no nontrivial automorphisms, and let~$\mathcal{L}$ be an upho lattice with core $L$. Let $\psi\colon \mathcal{L}\to \mathcal{L}$ be an automorphism of $\mathcal{L}$. We prove, by induction on rank, that $\psi$ is the identity. Assume that we have shown that $\psi$ acts as the identity on all elements of rank $\leq n$ for some $n \geq 0$. Let $x \in \mathcal{L}$ be any element of rank $n$, and let $y_1,\ldots,y_k$ be the covers of $x$. Since $\mathcal{L}$ is upho with core $L$, we have an isomorphisms $\varphi\colon [x,y_1\vee\cdots\vee y_k]\to L$. If $\psi$ nontrivially permuted the $y_1,\ldots,y_k$, then $\varphi \circ \psi \circ\varphi^{-1}$ would be a nontrivial automorphism of~$L$. So we must have $\varphi(y_i) = y_i$ for all these~$y_i$. But since $x$ was arbitrary, and every element of rank $n+1$ covers some element of rank $n$, we have shown that $\psi$ acts as the identity on rank $n$. By induction, we are done.
\end{proof}

Putting everything together, we have our main result of this section.

\begin{thm} \label{thm:finite}
Let $L$ be a finite graded lattice which has no nontrivial automorphisms. Then $L$ is the core of finitely many upho lattices, i.e., $\kappa(L)$ is finite.
\end{thm}

\begin{proof}
This follows from combining \cref{lem:autos}, \cref{cor:colorability}, and \cref{cor:monoid_lat_finiteness}.
\end{proof}

\begin{remark} \label{rem:atom_permute}
Inspecting the proof of \cref{lem:autos}, we see that it is in fact enough to assume the slightly weaker condition that $L$ has \emph{no automorphisms which nontrivially permute its atoms} to be able conclude that $\kappa(L)$ is finite. However, we phrased \cref{thm:finite} the way we did above because the statement is cleaner and because in practice for most interesting finite graded lattices~$L$, their automorphisms are determined by the way they act on the atoms. In particular this is true when $L$ is atomic, i.e., when every element is a join of some subset of atoms.
\end{remark}

For many classes of finite, combinatorial structures, ``most'' members have no nontrivial automorphisms, in the sense that the proportion of such structures on~$[n]$ with a nontrivial automorphism goes to $0$ as $n\to \infty$. This is known to be true for finite graphs~\cite{erdos1963asymmetric} and for finite posets~\cite{promel1987counting}, and we suspect it is true for finite lattices as well. Hence, \cref{thm:finite} should apply to most lattices $L$. On the other hand, we also expect~$\kappa(L) = 0$ for most $L$. So it is reasonable to ask whether there are~$L$ to which \cref{thm:finite} applies but for which we also know that~$\kappa(L) > 0$. We provide an infinite sequence of such~$L$, with both ranks and numbers of atoms going to infinity, in the following \cref{ex:weak_order}.

\begin{example} \label{ex:weak_order}
In \cite[\S4.3.1]{hopkins2024upho1} it is explained that the weak order of any finite Coxeter group is the core of an upho lattice, namely, the classical braid monoid. (The braid monoid gives an upho lattice because it is a Garside monoid~\cite{dehornoy1999gaussian, dehornoy2015foundations}.) For any Dynkin diagram that has no ``$\infty$'' labels, the automorphisms of the corresponding weak order are exactly the Dynkin diagram automorphisms: see~\cite[Corollary~3.2.6]{bjorner2005combinatorics}. Thus, for any $n > 2$, letting~$L$ be the weak order of the type~$B_n$ Coxeter group, $L$ has no nontrivial automorphisms. Hence for this finite graded lattice $L$, which has~$n$ atoms and rank~$n^2$, we have $0 < \kappa(L) < \infty$.
\end{example}

We conclude this section by briefly explaining, in the following \cref{ex:auto_still_finite}, how the techniques we developed here can sometimes be adapted to show that $\kappa(L)$ is finite for certain finite graded lattices $L$ that \emph{do} have nontrivial automorphisms.

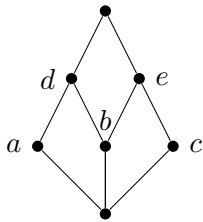
\begin{figure}
\begin{tikzpicture}[scale=0.9]
\node[circle,fill,inner sep=1.5pt] (1) at (0,0) {};
\node[circle,fill,inner sep=1.5pt, label=180:{$a$}] (2) at (-1,1) {};
\node[circle,fill,inner sep=1.5pt, label=90:{$b$}] (3) at (0,1) {};
\node[circle,fill,inner sep=1.5pt, label=0:{$c$}] (4) at (1,1) {};
\node[circle,fill,inner sep=1.5pt, label=180:{$d$}] (5) at (-0.5,2) {};
\node[circle,fill,inner sep=1.5pt, label=0:{$e$}] (6) at (0.5,2) {};
\node[circle,fill,inner sep=1.5pt] (7) at (0,3) {};
\draw (2) -- (1) -- (3) -- (1) -- (4);
\draw (2) -- (5) -- (3) -- (6) -- (4);
\draw (5) -- (7) -- (6);
\end{tikzpicture}
\caption{The lattice $L$ from \cref{ex:auto_still_finite} which has $\kappa(L)<\infty$ even though it has a nontrivial automorphism.} \label{fig:auto_still_finite}
\end{figure}

\begin{example} \label{ex:auto_still_finite}
 Consider the rank-three lattice $L$ depicted in \cref{fig:auto_still_finite}. This~$L$ has a nontrivial automorphism; it even has one which nontrivially permutes its atoms, as in~\cref{rem:atom_permute}. Nevertheless, it can be shown that $\kappa(L)$ is finite for this~$L$. The basic idea is to consider what an upho lattice $\mathcal{L}$ that has $L$ as a core must look like up to rank three, and conclude that any such $\mathcal{L}$ must have no nontrivial automorphisms. Indeed, in any such upho lattice $\mathcal{L}$ with core $L$, there must be a unique element that covers the atom $b$ and that does not belong to the core $L$; call it~$x$. By looking at the principal order filter generated by~$b$, which must be isomorphic to~$\mathcal{L}$ and so in particular have a copy of $L$ at the bottom, we see that for exactly one of the elements~$d$ or~$e$, the join of $x$ with this element has rank three. But this means that the way $\mathcal{L}$ looks up to rank three is asymmetric, and hence there are no automorphisms of $\mathcal{L}$ that nontrivially permute its atoms. By a similar argument as in the proof of \cref{lem:autos}, if~$\mathcal{L}$ has no automorphisms that nontrivially permute its atoms, it has no nontrivial automorphisms at all. Hence, by~\cref{cor:colorability} any such $\mathcal{L}$ is colorable, and therefore by  \cref{cor:monoid_lat_finiteness} we indeed have $\kappa(L) < \infty$. We also believe $\kappa(L) > 0$, because the monoid $M=\langle a,b,c \mid aa=bb, ba=ca\rangle$ should give an upho lattice with core $L$, conditional on the validity of \cref{conj:precoloring_monoid} below in this case.\footnote{This example was considered in \cite[Example~5.10]{hopkins2024upho1}, where the conjecture was implicitly assumed.}
\end{example}

\section{rank-two cores} \label{sec:rank_two}

In this section we explore rank-two cores. Rank-one cores are trivial: the only rank-one lattice is the two-element chain, and it is the core of a unique upho lattice~$\mathbb{N}$. But already rank-two cores are quite interesting. For each $n\geq 1$, there is a unique rank-two lattice with $n$ atoms, which we denote by $M_n$.\footnote{Please do not associate the ``M'' in the lattice $M_n$ with ``monoid;'' it stands rather for ``modular.'' We recall that a lattice $L$ is \dfn{modular} if $a \vee (x \wedge b) = (a \vee x) \wedge b$ for all $a,b,x\in L$ with $a \leq b$.} (In particular, $M_2=B_2$.) Our major result in this section is that for each $n \geq 2$, $\kappa(M_n)$ is greater than or equal to $p(n)$, the number of integer partitions of $n$. Along the way we also show that $\kappa(M_2)=2$. In fact, $M_2$ is the only nontrivial finite lattice where we completely understand all the ways it can be realized as a core of an upho lattice.

Fix $n\geq 2$. We start by describing two different recursive constructions that produce upho lattices with core~$M_n$. In both constructions we will build a sequence of finite posets $P_0 \subseteq P_1 \subseteq \cdots$, with $P_i$ of rank $i$, and then get an upho lattice as the union $\bigcup_{i=0}^{\infty} P_i$.

The first construction we call the \dfn{dominating vertex construction}, which produces a lattice we denote $\mathcal{D}_n$. We start by setting $P_0$ to be the one-element poset. Then, for each $i \geq 1$, we obtain $P_i$ from~$P_{i-1}$ by doing the following: 
\begin{itemize}
    \item we append a new element which covers all the elements of rank $i-1$;
    \item then for each element $p$ of rank $i-1$, we also append $n-1$ additional new elements covering only $p$.
\end{itemize}
We call this the ``dominating vertex construction'' because of the element of rank $i$ which covers (``dominates'') all elements of rank $i-1$. It is clear that $P_0 \subseteq P_1 \subseteq \cdots$, with $P_i$ of rank $i$, so that we can define $\mathcal{D}_n \coloneqq  \bigcup_{i=0}^{\infty} P_i$ to be the result of this construction. For example, the left side of \cref{fig:dominating_flip} depicts the first few ranks of~$\mathcal{D}_3$. 

The second construction we call the \dfn{flip construction}, which produces a lattice we denote~$\mathcal{F}_n$. We start by setting $P_0$ to be the one-element poset and $P_1$ to be the ``claw'' poset with minimum~$\hat{0}$, $n$ atoms, and no other elements. Then, for each $i \geq 2$, we obtain $P_i$ from $P_{i-1}$ by doing the following: 
\begin{itemize}
    \item for each element $p$ of rank $i-2$, letting $q_1,\ldots,q_n$ be the elements of rank $i-1$ covering $p$, we append a new element which covers exactly $q_1,\ldots,q_n$;
    \item then for each element $p$ of rank $i-1$, we also append enough additional new elements covering only $p$ to make $p$ be covered by exactly $n$ elements.
\end{itemize}
We call this the ``flip construction'' because the first step can be seen as taking the portion of the Hasse diagram between ranks $i-2$ and $i-1$ and placing a reflected copy above it. It is again clear that $P_0 \subseteq P_1 \subseteq \cdots$, with $P_i$ of rank $i$, so that we can define $\mathcal{F}_n \coloneqq  \bigcup_{i=0}^{\infty} P_i$ to be the result of this construction. For example, the right side of \cref{fig:dominating_flip} depicts the first few ranks of $\mathcal{F}_3$. 

\begin{figure}
    \begin{center}
    \begin{tikzpicture}[xscale=0.4,yscale=0.7]
    \node[circle,fill=black,inner sep=0pt] (A) at (0,0) {};
    
    \node[circle,fill=black,inner sep=0pt] (B) at (-1.5,1){};
    \node[circle,fill=black,inner sep=0pt] (C) at (0,1){};
    \node[circle,fill=black,inner sep=0pt] (D) at (1.5,1){};

    \node[circle,fill=black,inner sep=0pt] (E) at (-4.5,2){};
    \node[circle,fill=black,inner sep=0pt] (F) at (-3,2){};
    \node[circle,fill=black,inner sep=0pt] (G) at (-1.5,2){};
    \node[circle,fill=black,inner sep=0pt] (H) at (0,2){};
    \node[circle,fill=black,inner sep=0pt] (I) at (1.5,2){};
    \node[circle,fill=black,inner sep=0pt] (J) at (3,2){};
    \node[circle,fill=black,inner sep=0pt] (K) at (4.5,2){};

    \node[circle,fill=black,inner sep=0pt] (L) at (-7,3){};
    \node[circle,fill=black,inner sep=0pt] (M) at (-6,3){};
    \node[circle,fill=black,inner sep=0pt] (N) at (-5,3){};
    \node[circle,fill=black,inner sep=0pt] (O) at (-4,3){};
    \node[circle,fill=black,inner sep=0pt] (P) at (-3,3){};
    \node[circle,fill=black,inner sep=0pt] (Q) at (-2,3){};
    \node[circle,fill=black,inner sep=0pt] (R) at (-1,3){};
    \node[circle,fill=black,inner sep=0pt] (S) at (0,3){};
    \node[circle,fill=black,inner sep=0pt] (T) at (1,3){};
    \node[circle,fill=black,inner sep=0pt] (U) at (2,3){};
    \node[circle,fill=black,inner sep=0pt] (V) at (3,3){};
    \node[circle,fill=black,inner sep=0pt] (W) at (4,3){};
    \node[circle,fill=black,inner sep=0pt] (X) at (5,3){};
    \node[circle,fill=black,inner sep=0pt] (Y) at (6,3){};
    \node[circle,fill=black,inner sep=0pt] (Z) at (7,3){};

    \node[circle,fill=black,inner sep=0pt] (1) at (-7.5,3.75){};
    \node[circle,fill=black,inner sep=0pt] (2) at (-7,3.75){};
    \node[circle,fill=black,inner sep=0pt] (3) at (-6.5,3.75){};
    \node[circle,fill=black,inner sep=0pt] (4) at (-6,3.75){};
    \node[circle,fill=black,inner sep=0pt] (5) at (-5.5,3.75){};
    \node[circle,fill=black,inner sep=0pt] (6) at (-5,3.75){};
    \node[circle,fill=black,inner sep=0pt] (7) at (-4.5,3.75){};
    \node[circle,fill=black,inner sep=0pt] (8) at (-4,3.75){};
    \node[circle,fill=black,inner sep=0pt] (9) at (-3.5,3.75){};
    \node[circle,fill=black,inner sep=0pt] (10) at (-3,3.75){};
    \node[circle,fill=black,inner sep=0pt] (11) at (-2.5,3.75){};
    \node[circle,fill=black,inner sep=0pt] (12) at (-2,3.75){};
    \node[circle,fill=black,inner sep=0pt] (13) at (-1.5,3.75){};
    \node[circle,fill=black,inner sep=0pt] (14) at (-1,3.75){};
    \node[circle,fill=black,inner sep=0pt] (15) at (-0.5,3.75){};
    \node[circle,fill=black,inner sep=0pt] (16) at (0,3.75){};
    \node[circle,fill=black,inner sep=0pt] (17) at (0.5,3.75){};
    \node[circle,fill=black,inner sep=0pt] (18) at (1,3.75){};
    \node[circle,fill=black,inner sep=0pt] (19) at (1.5,3.75){};
    \node[circle,fill=black,inner sep=0pt] (20) at (2,3.75){};
    \node[circle,fill=black,inner sep=0pt] (21) at (2.5,3.75){};
    \node[circle,fill=black,inner sep=0pt] (22) at (3,3.75){};
    \node[circle,fill=black,inner sep=0pt] (23) at (3.5,3.75){};
    \node[circle,fill=black,inner sep=0pt] (24) at (4,3.75){};
    \node[circle,fill=black,inner sep=0pt] (25) at (4.5,3.75){};
    \node[circle,fill=black,inner sep=0pt] (26) at (5,3.75){};
    \node[circle,fill=black,inner sep=0pt] (27) at (5.5,3.75){};
    \node[circle,fill=black,inner sep=0pt] (28) at (6,3.75){};
    \node[circle,fill=black,inner sep=0pt] (29) at (6.5,3.75){};
    \node[circle,fill=black,inner sep=0pt] (30) at (7,3.75){};
    \node[circle,fill=black,inner sep=0pt] (31) at (7.5,3.75){};

    \draw (A) -- (B);
    \draw (A) -- (C);
    \draw (A) -- (D);

    \draw (B) -- (E);
    \draw (B) -- (F);
    \draw (B) -- (H);
    \draw (C) -- (G);
    \draw (C) -- (H);
    \draw (C) -- (I);
    \draw (D) -- (H);
    \draw (D) -- (J);
    \draw (D) -- (K);

    \draw (E) -- (L);
    \draw (E) -- (M);
    \draw (E) -- (S);
    \draw (F) -- (N);
    \draw (F) -- (O);
    \draw (F) -- (S);
    \draw (G) -- (P);
    \draw (G) -- (Q);
    \draw (G) -- (S);
    \draw (H) -- (R);
    \draw (H) -- (S);
    \draw (H) -- (T);
    \draw (I) -- (S);
    \draw (I) -- (U);
    \draw (I) -- (V);
    \draw (J) -- (S);
    \draw (J) -- (W);
    \draw (J) -- (X);
    \draw (K) -- (S);
    \draw (K) -- (Y);
    \draw (K) -- (Z);

    \draw (L) -- (1);
    \draw (L) -- (2);
    \draw (L) -- (16);
    \draw (M) -- (3);
    \draw (M) -- (4);
    \draw (M) -- (16);
    \draw (N) -- (5);
    \draw (N) -- (6);
    \draw (N) -- (16);
    \draw (O) -- (7);
    \draw (O) -- (8);
    \draw (O) -- (16);
    \draw (P) -- (9);
    \draw (P) -- (10);
    \draw (P) -- (16);
    \draw (Q) -- (11);
    \draw (Q) -- (12);
    \draw (Q) -- (16);
    \draw (R) -- (13);
    \draw (R) -- (14);
    \draw (R) -- (16);
    \draw (S) -- (15);
    \draw (S) -- (16);
    \draw (S) -- (17);
    \draw (T) -- (16);
    \draw (T) -- (18);
    \draw (T) -- (19);
    \draw (U) -- (16);
    \draw (U) -- (20);
    \draw (U) -- (21);
    \draw (V) -- (16);
    \draw (V) -- (22);
    \draw (V) -- (23);
    \draw (W) -- (16);
    \draw (W) -- (24);
    \draw (W) -- (25);
    \draw (X) -- (16);
    \draw (X) -- (26);
    \draw (X) -- (27);
    \draw (Y) -- (16);
    \draw (Y) -- (28);
    \draw (Y) -- (29);
    \draw (Z) -- (16);
    \draw (Z) -- (30);
    \draw (Z) -- (31);
    \end{tikzpicture}
    \qquad \vrule \qquad
    \begin{tikzpicture}[xscale=0.4,yscale=0.7]
    \node[circle,fill=black,inner sep=0pt] (A) at (0,0) {};

    \node[circle,fill=black,inner sep=0pt] (B) at (-1.5,1){};
    \node[circle,fill=black,inner sep=0pt] (C) at (0,1){};
    \node[circle,fill=black,inner sep=0pt] (D) at (1.5,1){};

    \node[circle,fill=black,inner sep=0pt] (E) at (-4.5,2){};
    \node[circle,fill=black,inner sep=0pt] (F) at (-3,2){};
    \node[circle,fill=black,inner sep=0pt] (G) at (-1.5,2){};
    \node[circle,fill=black,inner sep=0pt] (H) at (0,2){};
    \node[circle,fill=black,inner sep=0pt] (I) at (1.5,2){};
    \node[circle,fill=black,inner sep=0pt] (J) at (3,2){};
    \node[circle,fill=black,inner sep=0pt] (K) at (4.5,2){};

    \node[circle,fill=black,inner sep=0pt] (L) at (-7,3){};
    \node[circle,fill=black,inner sep=0pt] (M) at (-6,3){};
    \node[circle,fill=black,inner sep=0pt] (N) at (-5,3){};
    \node[circle,fill=black,inner sep=0pt] (O) at (-4,3){};
    \node[circle,fill=black,inner sep=0pt] (P) at (-3,3){};
    \node[circle,fill=black,inner sep=0pt] (Q) at (-2,3){};
    \node[circle,fill=black,inner sep=0pt] (R) at (-1,3){};
    \node[circle,fill=black,inner sep=0pt] (S) at (0,3){};
    \node[circle,fill=black,inner sep=0pt] (T) at (1,3){};
    \node[circle,fill=black,inner sep=0pt] (U) at (2,3){};
    \node[circle,fill=black,inner sep=0pt] (V) at (3,3){};
    \node[circle,fill=black,inner sep=0pt] (W) at (4,3){};
    \node[circle,fill=black,inner sep=0pt] (X) at (5,3){};
    \node[circle,fill=black,inner sep=0pt] (Y) at (6,3){};
    \node[circle,fill=black,inner sep=0pt] (Z) at (7,3){};

    \node[circle,fill=black,inner sep=0pt] (1) at (-7.5,4){};
    \node[circle,fill=black,inner sep=0pt] (2) at (-7,4){};
    \node[circle,fill=black,inner sep=0pt] (3) at (-6.5,4){};
    \node[circle,fill=black,inner sep=0pt] (4) at (-6,4){};
    \node[circle,fill=black,inner sep=0pt] (5) at (-5.5,4){};
    \node[circle,fill=black,inner sep=0pt] (6) at (-5,4){};
    \node[circle,fill=black,inner sep=0pt] (7) at (-4.5,4){};
    \node[circle,fill=black,inner sep=0pt] (8) at (-4,4){};
    \node[circle,fill=black,inner sep=0pt] (9) at (-3.5,4){};
    \node[circle,fill=black,inner sep=0pt] (10) at (-3,4){};
    \node[circle,fill=black,inner sep=0pt] (11) at (-2.5,4){};
    \node[circle,fill=black,inner sep=0pt] (12) at (-2,4){};
    \node[circle,fill=black,inner sep=0pt] (13) at (-1.5,4){};
    \node[circle,fill=black,inner sep=0pt] (14) at (-1,4){};
    \node[circle,fill=black,inner sep=0pt] (15) at (-0.5,4){};
    \node[circle,fill=black,inner sep=0pt] (16) at (0,4){};
    \node[circle,fill=black,inner sep=0pt] (17) at (0.5,4){};
    \node[circle,fill=black,inner sep=0pt] (18) at (1,4){};
    \node[circle,fill=black,inner sep=0pt] (19) at (1.5,4){};
    \node[circle,fill=black,inner sep=0pt] (20) at (2,4){};
    \node[circle,fill=black,inner sep=0pt] (21) at (2.5,4){};
    \node[circle,fill=black,inner sep=0pt] (22) at (3,4){};
    \node[circle,fill=black,inner sep=0pt] (23) at (3.5,4){};
    \node[circle,fill=black,inner sep=0pt] (24) at (4,4){};
    \node[circle,fill=black,inner sep=0pt] (25) at (4.5,4){};
    \node[circle,fill=black,inner sep=0pt] (26) at (5,4){};
    \node[circle,fill=black,inner sep=0pt] (27) at (5.5,4){};
    \node[circle,fill=black,inner sep=0pt] (28) at (6,4){};
    \node[circle,fill=black,inner sep=0pt] (29) at (6.5,4){};
    \node[circle,fill=black,inner sep=0pt] (30) at (7,4){};
    \node[circle,fill=black,inner sep=0pt] (31) at (7.5,4){};

    \draw (A) -- (B);
    \draw (A) -- (C);
    \draw (A) -- (D);

    \draw (B) -- (E);
    \draw (B) -- (F);
    \draw (B) -- (H);
    \draw (C) -- (G);
    \draw (C) -- (H);
    \draw (C) -- (I);
    \draw (D) -- (H);
    \draw (D) -- (J);
    \draw (D) -- (K);

    \draw (E) -- (L);
    \draw (E) -- (M);
    \draw (E) -- (O);
    \draw (F) -- (N);
    \draw (F) -- (O);
    \draw (F) -- (P);
    \draw (G) -- (Q);
    \draw (G) -- (R);
    \draw (G) -- (S);
    \draw (H) -- (O);
    \draw (H) -- (S);
    \draw (H) -- (W);
    \draw (I) -- (S);
    \draw (I) -- (T);
    \draw (I) -- (U);
    \draw (J) -- (V);
    \draw (J) -- (W);
    \draw (J) -- (X);
    \draw (K) -- (W);
    \draw (K) -- (Y);
    \draw (K) -- (Z);

    \draw (L) -- (1);
    \draw (L) -- (2);
    \draw (M) -- (3);
    \draw (M) -- (5);
    \draw (N) -- (6);
    \draw (N) -- (7);
    \draw (P) -- (9);
    \draw (P) -- (10);
    \draw (Q) -- (11);
    \draw (Q) -- (12);
    \draw (R) -- (13);
    \draw (R) -- (15);
    \draw (T) -- (17);
    \draw (T) -- (19);
    \draw (U) -- (20);
    \draw (U) -- (21);
    \draw (V) -- (22);
    \draw (V) -- (23);
    \draw (X) -- (25);
    \draw (X) -- (26);
    \draw (Y) -- (27);
    \draw (Y) -- (29);
    \draw (Z) -- (30);
    \draw (Z) -- (31);

    \draw (L) -- (4);
    \draw (M) -- (4);
    \draw (O) -- (4);
    \draw (N) -- (8);
    \draw (O) -- (8);
    \draw (P) -- (8);
    \draw (Q) -- (14);
    \draw (R) -- (14);
    \draw (S) -- (14);
    \draw (O) -- (16);
    \draw (S) -- (16);
    \draw (W) -- (16);
    \draw (S) -- (18);
    \draw (T) -- (18);
    \draw (U) -- (18);
    \draw (V) -- (24);
    \draw (W) -- (24);
    \draw (X) -- (24);
    \draw (W) -- (28);
    \draw (Y) -- (28);
    \draw (Z) -- (28);
    \end{tikzpicture}
    \end{center}
    \caption{On the left, the dominating vertex construction of an upho lattice $\mathcal{D}_3$ with core~$M_3$; and on the right, the flip construction of an upho lattice $\mathcal{F}_3$ with core~$M_3$.} \label{fig:dominating_flip}
\end{figure}
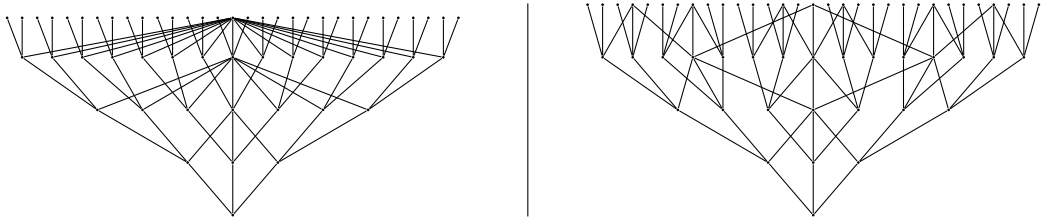

\begin{thm} \label{thm:dv_flip_lattices}
For any $n \geq 2$, the dominating vertex and flip constructions produce two different upho lattices $\mathcal{D}_n$ and $\mathcal{F}_n$ with core $M_n$.
\end{thm}

\begin{remark}
Both of these upho lattices already appeared in~\cite{hopkins2024upho1}, albeit not as explicitly. Namely, $\mathcal{D}_n$ is the upho lattice in~\cite[Theorem~4.2]{hopkins2024upho1}, and $\mathcal{F}_n$ is the upho lattice for the dual braid monoid of a rank two Coxeter group as in~\cite[\S4.3.2]{hopkins2024upho1}.
\end{remark}

Although it is not too difficult to prove \cref{thm:dv_flip_lattices} directly, we postpone the proof until we have described a more general construction of upho lattices with core $M_n$ coming from monoids. But let us point out right now that if an upho lattice agrees with either of these constructions up to rank three, it agrees forever.

\begin{thm} \label{thm:dv_flip_unique}
Let $\mathcal{L}$ be an upho lattice with core $M_n$ for some $n \geq 2$. If~$\mathcal{L}$ agrees with $\mathcal{D}_n$ up to rank three, it must in fact be $\mathcal{D}_n$ (up to isomorphism). The same is true for $\mathcal{F}_n$.
\end{thm}

\begin{proof}
We first consider the dominating vertex construction.  Fix an integer $k \geq 3$, and suppose that $\mathcal{L}$ is an upho lattice with core $M_n$ and that for each~$i \leq k$, there is an element of rank $i$ in $\mathcal{L}$ that covers all elements of rank $i - 1$. Note that this implies that $\mathcal{L}$ agrees with~$\mathcal{D}_n$ up to rank $k$. Indeed, each element at rank $i$ other than the ``dominating'' one covers exactly one element of rank $i-1$: if another rank-$i$ element covered two or more elements of rank $i-1$, then these rank-$(i-1)$ elements would have incomparable upper bounds at rank $i$; and each element of rank~$i$ must cover an element of rank $i-1$ for the lattice to be graded. Moreover, each element at rank $i-1$ is covered by $n$ elements: one is the ``dominating'' element of rank $i$ and the other $n-1$ have no other lower covers. 

Let $s_1$, \ldots, $s_n$ be the atoms of $\mathcal{L}$.  For any~$1 \leq i \leq n$, consider the order filter $V_{s_i}$ above $s_i$: it contains some set $S_i$ of vertices at rank $k$ in~$\mathcal{L}$.  These vertices are at rank $k - 1$ in $V_{s_i}$, and since $V_{s_i} \cong \mathcal{L}$, we have by assumption that there is an element $x_i$ in $V_{s_i}$ that covers all vertices in~$S_i$.

Let $u = s_1 \vee \cdots \vee s_n$ be the element of rank $2$ in $\mathcal{L}$ that covers all the $s_i$.  Consider the set $S$ of vertices at rank $k$ in $\mathcal{L}$ that are larger than $u$. Obviously, $S \subseteq S_i$ for all~$1 \leq i \leq n$, and hence every element in $S$ is covered by $x_1, \ldots, x_n$.  Since $k > 2$ and $\mathcal{L}$ is upho and not a chain, $\#S \geq 2$.  That is, $S$ is a set of at least $2$ vertices in $\mathcal{L}$, all of which are covered by all of $x_1, \ldots, x_n$.  Since $\mathcal{L}$ is a lattice, this is only possible if $x_1 = \cdots = x_n$, and this single element covers all of the vertices in all of the $S_i$.  Since every element at rank $k$ in $\mathcal{L}$ must be larger than some atom $s_i$, and therefore must belong to one of the sets $S_i$, it follows that the element $x_1$ covers all elements of rank $k$ in $\mathcal{L}$, so that $\mathcal{L}$ agrees with the $\mathcal{D}_n$ through rank $k + 1$.  The result now follows by induction.

We now consider the flip construction.  Observe that, by construction, every element of $\mathcal{F}_n$ either covers exactly $n$ elements (if it is the flip of an element two ranks lower) or exactly $1$ element (otherwise).   Now fix an integer $k \geq 3$, and suppose, for sake of contradiction, that $\mathcal{L}$ is an upho lattice with core $M_n$ that agrees with the flip construction up to rank $k$, but does not agree with it at rank $k + 1$.  Then there is an element $x$ at rank $k + 1$ in $\mathcal{L}$ such that there are elements $y, z$ with $y \lessdot x$ and $z \lessdot x$ but $y$ and $z$ do not cover an element in common. Let $y'$ be an element covered by $y$ and let $z'$ be an element covered by $z$.  We consider two possibilities.

First, suppose that $y'$ and $z'$ are covered in common by some element $w \coloneqq y' \vee z'$, i.e., that $\rho(y' \vee z') = k$.   Since $\rho(w) = k$,  $\mathcal{L}$ agrees with $\mathcal{F}_n$ through rank $k$, and~$w$ covers more than one element, $w$ is a flip of some vertex two ranks below it, and all the elements covered by $w$ share a single lower cover.  Therefore $y' \wedge z'$ has rank~$k - 2 > 0$.  However, in this case the elements $x, y, z$ all belong to $V_{y' \wedge z'} \cong \mathcal{L}$, and have rank~$\leq 3$ in this poset; this contradicts the hypothesis that $\mathcal{L}$ and $\mathcal{F}_n$ agree to rank $k$.  Therefore this case does not occur.

Alternatively, it must be the case that for every possible choice of $x, y, z, y', z'$ in~$\mathcal{L}$ such that $y' \lessdot y \lessdot x \gtrdot z \gtrdot z'$ for which $y$ and $z$ do not cover anything in common, there is no element that covers both $y'$ and $z'$.  For $i \geq 0$, let $a_i$ be the size of the $i$th rank of $\mathcal{F}_n$.  By \cref{cor:upho_gfs}, this is also the size of the $i$th rank of~$\mathcal{L}$. Moreover, these numbers satisfy the recurrence relation $a_{i} = n a_{i - 1} - (n - 1)a_{i - 2}$ for all $i \geq 2$.  Now let us count cover relations between ranks $k$ and $k + 1$, which we refer to as ``edges,'' because they are edges in the Hasse diagram of $\mathcal{L}$.  On one hand, since $\mathcal{L}$ is upho with core $M_n$, there are $n \cdot a_k$ of these edges.  On the other hand, it follows from the defining hypothesis of this case that~$n \cdot a_{k - 1}$ of these edges belong to copies of $M_n$ with minimum elements at rank~$k - 1$ (i.e., all those subgraphs of the Hasse diagram are edge-disjoint), so that there are $n \cdot a_k - n \cdot a_{k - 1} = a_{k + 1} - a_{k - 1}$ other edges.  Observe also that there are at most $a_{k - 1}$ elements at rank $k + 1$ that are the maximum elements of the copies of $M_n$ whose minimum elements lie at rank~$k - 1$, and strictly fewer than $a_{k - 1}$ if some choice of $y$ and $z$ lie in two of these copies of $M_n$. Therefore, there are at least $a_{k + 1} - a_{k - 1}$ ``other'' elements, and again strictly more under the same condition.  If the inequalities are strict, this leads to a contradiction because there are more than $a_{k + 1} - a_{k - 1}$ ``other'' elements at rank $k + 1$, each of which covers at least one element, but only $a_{k + 1} - a_{k - 1}$ edges to connect them to. If the inequality is actually equality, then we instead get a contradiction because the~$x$ corresponding to our choice of $y$ and $z$, which covers at least two elements, must be among the ``other'' elements, and so there are still not enough edges available.

In any case we arrive at a contradiction from the assumption that $\mathcal{L}$ disagrees with $\mathcal{F}_n$ at rank $k + 1$. Therefore, it does agree, and the result follows by induction.
\end{proof}

\begin{remark}
From \cref{thm:dv_flip_unique} one can deduce that, for any $n \geq 2$, the flip construction $\mathcal{F}_n$ is the only \emph{modular} upho lattice with core $M_n$.
\end{remark}

These two constructions yield all possible upho lattices with core~$M_2$.

\begin{cor}
The only upho lattices with core $M_2=B_2$ are the the results of the dominating vertex and flip constructions, $\mathcal{D}_2$ and $\mathcal{F}_2$. (These are depicted in \cref{fig:b2_ex}, with $\mathcal{F}_2$ on the left and $\mathcal{D}_2$ on the right.) Hence, we have $\kappa(M_2)=2$.
\end{cor}

\begin{proof}
Let $\mathcal{L}$ be an upho lattice with core $M_2$. By \cref{cor:upho_gfs}, $\mathcal{L}$ has $n+1$ elements at rank $n$ for all $n \geq 0$. Also, every element in $\mathcal{L}$ is covered by exactly two elements in the rank above, and covers at least one element in the rank below. Using these facts, it is routine to check that the only possibilities for what $\mathcal{L}$ could look like up to rank three are $\mathcal{D}_2$ or $\mathcal{F}_2$. The statement then follows from \cref{thm:dv_flip_unique}.
\end{proof}

However, for $n \geq 3$, there are more possibilities beyond these two for upho lattices with core $M_n$. In order to understand these possibilities, and in order to prove \cref{thm:dv_flip_lattices} as promised, we need to return to monoids. The following theorem gives us a rich source of upho lattices with core~$M_n$.

\begin{thm} \label{thm:rank2_monoids}
Let $n \geq 2$ and let $f\colon [n] \to [n]$ be any function. Define the homogeneously finitely generated monoid $M(f)$ by
\[ M(f) \coloneqq \langle s_1,\ldots,s_n\mid s_1s_{f(1)}=s_2s_{f(2)}= \cdots = s_ns_{f(n)}\rangle.\]
Then $M(f)$ is left-cancellative and any two elements in $M(f)$ have a least common right multiple. Hence, $\mathcal{L}(f)\coloneqq (M(f),\leq_L)$ is an upho lattice, with core $M_n$.
\end{thm}

Before we prove \cref{thm:rank2_monoids}, let us explain how \cref{thm:dv_flip_lattices} is an easy corollary of \cref{thm:rank2_monoids}.

\begin{proof}[Proof of \cref{thm:dv_flip_lattices}, assuming \cref{thm:rank2_monoids}]
The results of both the dominating vertex and flip constructions can be obtained as $\mathcal{L}(f)$ from particularly nice choices of $f$ in \cref{thm:rank2_monoids}. Indeed, these two constructions correspond to the two ``extreme'' cases of functions $f\colon [n]\to [n]$ in terms of their fiber structures.

Choosing any $f\colon [n] \to [n]$ which has image of size $1$ yields the dominating vertex construction $\mathcal{D}_n$ for $\mathcal{L}(f)$. To see this, for convenience consider the function $f$ defined by $f(1) = f(2) =\cdots = f(n) = n$. Let $S = \{s_1,\ldots,s_n\}$ and again let $S^*$ denote the set of words over the alphabet $S$. Then, for any two words $v,w \in S^*$ of the same length, we have $vs_n=ws_n$ in $M(f)$, while $vs_i$ and $ws_j$ will be distinct, and be distinct from~$vs_n$, for all $1 \leq i,j \leq n-1$. Thus, $\mathcal{L}(f)$ is indeed~$\mathcal{D}_n$.

On the other hand, choosing any bijective $f\colon [n] \to [n]$ yields the flip construction~$\mathcal{F}_n$ for $\mathcal{L}(f)$. To see this, for convenience consider the identity function $f$ on~$[n]$.  By the symmetry with respect to left and right multiplication of the defining relations for this $f$, we have that $M(f)$ is not just left-cancellative, but also \emph{right-cancellative}. So now suppose that $v,w \in M(f)$ are two elements which (when viewed as words in~$S^*$) have length $k$, and which satisfy~$vs_i=ws_j$ for some $1\leq i,j \leq n$. There are two possibilities. If $i=j$, then by right-cancellativity we have $v=w$. If $i\neq j$, then we claim that there is $u\in M(f)$ of length~$k-1$ such that $v=us_i$ and $w=us_j$. Indeed, to convert the final $s_i$ in $vs_i$ to an $s_j$, we must have $vs_i=us_is_i=us_js_j$ for some $u \in M(f)$ of length $k-1$. Then since $ws_j  = vs_i = us_js_j$, by right-cancellativity again we get $v=us_i$ and $w=us_j$, as claimed. Thus two elements of the same rank in $\mathcal{L}(f)$ have a common cover exactly when they cover something in common, and so $\mathcal{L}(f)$ is indeed $\mathcal{F}_n$.
\end{proof}

In order to prove \cref{thm:rank2_monoids}, we need some preparatory lemmas. First, to prove that the monoid $M(f)$ in \cref{thm:rank2_monoids} is left-cancellative, we will use the following simple lemma.

\begin{lemma} \label{lem:rels_left_cancel}
Let $M=\langle S \mid R \rangle$ be a homogeneously finitely generated monoid such that for each generator $s\in S$, there is at most one word that begins with $s$ which appears in the relations in~$R$. Then $M$ is left-cancellative.
\end{lemma}

\begin{proof}
To prove that the monoid $M$ is left-cancellative, a standard inductive argument implies it is enough to show that whenever $sb=sc$ for $b,c\in M$ arbitrary elements and $s\in S$ a generator, then we have $b=c$. So suppose that $b,c \in M$ and~$s\in S$ satisfy $sb=sc$. Because the relations defining $M$ are homogeneous, it must be that~$b$ and $c$ (when viewed as words in~$S^*$) have the same length $n \geq 0$. We will prove that $b=c$ by induction on $n\geq 0$. If~$n=0$ the claim is trivial, so suppose that~$n \geq 1$ and that the claim has been proved for smaller values of $n$. By the same standard inductive argument, knowing the claim for smaller values of $n$ also means that whenever $ab=ac$ for \emph{any} elements $a,b,c \in M$ for which the length of $ab$ (which must be the length of $ac$) is less than $n$, we have $b=c$.

By supposition, we can convert $sb$ (viewed as a word in $S^*$) to $sc$ by applying a series of the relations in $R$. If when applying these relations, we never change the first letter, then clearly $b=c$ in $M$. So suppose that we do change the first letter at some point. For a generator $t \in S$, let $w_t$ denote the unique word (if it exists) for which $tw_t$ appears in the relations in~$R$. Then set $s_0 \coloneqq s$, and suppose that as we convert~$sb$ to~$sc$, the first time we change the first letter is~$s_0w_{s_0}v_0 = s_1w_{s_1}v_0$ for some other generator~$s_1\in S$ and some word $v_0\in S^*$. Notice that we have~$b=w_{s_0}v_0$ in $M$ because by supposition we never modified the first letter before getting to $s_0w_{s_0}v_0$. Then suppose the next time we change the first letter, it is $s_1w_{s_1}v_1=s_2w_{s_2}v_1$ for some $s_2\in S$, $v_1\in S^*$. The same reasoning as before says that $w_{s_1}v_0=w_{s_1}v_1$ in $M$. By the inductive hypothesis, because $w_{s_1}v_0$ has length $n-1$, this actually means that~$v_0=v_1$ in $M$. We continue defining $s_i\in S$ and $v_i \in S^*$ in the same fashion, until the last time we change the first letter, say to $s_kw_{s_k}v_{k-1}$. By the same inductive argument repeated $k - 1$ times, we have $v_0=v_1=\cdots=v_{k-1}$ in $M$. Notice also that we must have $s_k=s$, and that~$w_{s_k}v_{k-1}=c$ in $M$ because after~$s_kw_{s_k}v_{k-1}$ we never change the first letter. But this means that $b=w_{s}v_0=w_sv_{k-1}=c$, as claimed.
\end{proof}

Next, to prove that the monoid $M(f)$ in \cref{thm:rank2_monoids} has least common multiples, we will use the following variant of a lemma due to Bj\"{o}rner--Edelman--Ziegler~\cite{bjorner1990hyperplane} which says that this lattice property can be checked locally.

\begin{lemma}[{cf.~\cite[Lemma~2.1]{bjorner1990hyperplane}}] \label{lem:bez}
Let $\mathcal{P}$ be a finite-type $\mathbb{N}$-graded poset. Suppose that for any $x, y\in\mathcal{P}$, $x$ and $y$ have an upper bound, and moreover, if $x$ and $y$ both cover an element $z$, then they have a least upper bound $x \vee y$. Then~$\mathcal{P}$ is a lattice.
\end{lemma}

\begin{proof}
As indicated, this is essentially a special case of the famous ``BEZ lemma''~\cite[Lemma~2.1]{bjorner1990hyperplane}. Let $\mathcal{P}$ be as in the statement of the lemma. We first show that for any~$x,y \in \mathcal{P}$, their join $x \vee y$ exists in $\mathcal{P}$. By assumption, an upper bound for $x$ and $y$ exists, call it $z \in \mathcal{P}$. Now, $[\hat{0},z]$ is a finite poset, and it has the property that for any pair of elements covering a common element, their join exists, so the BEZ lemma applies and says that $[\hat{0},z]$ is a lattice. In particular, a join of $x$ and $y$ exists in $[\hat{0},z]$, call it $w$. Now, we claim this $w$ is actually the least upper bound of $x$ and~$y$ in $\mathcal{P}$ as well. Indeed, because $[\hat{0},z]$ is an interval, it is clear that there cannot be an upper bound in $\mathcal{P}$ for $x$ and $y$ that is smaller than $w$. And so if $z' \in \mathcal{P}$ is any upper bound for $x$ and $y$, then, letting $z''$ be an upper bound for $w$ and $z'$, the BEZ lemma applies also to $[\hat{0},z'']$ to show that indeed $w \leq z'$. 

Having shown that joins of all pairs of elements exist, it is easy to show that meets exist, i.e., that $\mathcal{P}$ is a lattice. Let $x,y \in \mathcal{P}$. Then $[\hat{0},x\vee y]$ is a finite join semilattice with a minimum, hence is a lattice by~\cite[Proposition~3.3.1]{stanley2012ec1}. So the meet $x \wedge y$ exists in $[\hat{0},x\vee y]$, and this must be their meet in $\mathcal{P}$ as well.
\end{proof}

With \cref{lem:rels_left_cancel} and \cref{lem:bez}, we are now ready to prove \cref{thm:rank2_monoids}.

\begin{proof}[Proof of \cref{thm:rank2_monoids}]
By construction, $M(f)$ is a homogeneously finitely generated monoid. Moreover, it clearly satisfies the condition of~\cref{lem:rels_left_cancel}, so it is left-cancellative. Hence, by~\cref{lem:monoids}, $\mathcal{L}(f)$ is a (finite-type $\mathbb{N}$-graded) upho poset. What remains is to show that $\mathcal{L}(f)$ is a lattice.

First, let us show that for every $v,w \in \mathcal{L}(f)$, they have some upper bound. It suffices to prove this for $v,w \in M(f)$ which have the same length (when viewed as words in $S^*$, where $S=\{s_1,\ldots,s_n\}$), because if one has a shorter length than the other we can just multiply the shorter one on the right by some generators. So suppose $v=s_{i_1}\cdots s_{i_k}$ and $w=s_{j_1}\cdots s_{j_k}$. Then we claim that
\[s_{i_1}\cdots s_{i_k}s_{f(i_k)}s_{f^3(i_{k-1})}\cdots s_{f^{2k-1}(i_1)} = s_{j_1}\cdots s_{j_k}s_{f(j_k)}s_{f^3(j_{k-1})}\cdots s_{f^{2k-1}(j_1)}.\]
We prove this claim by induction on $k$. By successively applying relations in $M(f)$, we have that $s_{i_1}\cdots s_{i_k}s_{f(i_k)}\cdots s_{f^{2k-1}(i_1)}$ is equal to
\begin{gather*}
s_{i_1}\cdots s_{i_{k-1}} (s_{i_k}s_{f(i_k)})s_{f^3(i_{k-1})}\cdots s_{f^{2k-1}(i_1)} \\
= s_{i_1}\cdots s_{i_{k-1}}s_{f(i_{k-1})}s_{f^2(i_{k-1})}s_{f^3(i_{k-1})}\cdots s_{f^{2k-1}(i_1)}\\
= s_{i_1}\cdots s_{f(i_{k-2})}s_{f^2(i_{k-2})}s_{f^2(i_{k-1})}s_{f^3(i_{k-1})}\cdots s_{f^{2k-1}(i_1)}\\
\vdots\\
= s_{i_1}s_{f(i_1)}s_{f^2(i_1)}s_{f^2(i_2)}\cdots s_{f^2(i_{k-2})}s_{f^2(i_{k-1})}s_{f^3(i_{k-1})}\cdots s_{f^{2k-1}(i_1)}\\
= s_{j_1}s_{f(j_1)}s_{f^2(i_1)}s_{f^2(i_2)}\cdots s_{f^2(i_{k-2})}s_{f^2(i_{k-1})}s_{f^3(i_{k-1})}\cdots s_{f^{2k-1}(i_1)} \,.
\end{gather*}
The suffix of this word of length $2k-2$ is of the form described by the claim, so by induction it is equal to
\[s_{f^2(j_1)}s_{f^2(j_2)}\cdots s_{f^2(j_{k-2})}s_{f^2(j_{k-1})}s_{f^3(j_{k-1})}\cdots s_{f^{2k-1}(j_1)} \,.\]
Then reversing the steps in the prior chain of equalities (now with $j$'s instead of $i$'s), we conclude that $s_{i_1} \cdots s_{i_k} s_{f(i_k)}\cdots s_{f^{2k-1}(i_1)}$ is equal to
\begin{gather*}
s_{j_1} s_{f(j_1)} s_{f^2(i_1)} s_{f^2(i_2)}\cdots s_{f^2(i_{k-2})} s_{f^2(i_{k-1})} s_{f^3(i_{k-1})} \cdots s_{f^{2k-1}(i_1)} \\
=s_{j_1} s_{f(j_1)} s_{f^2(j_1)}s_{f^2(j_2)} \cdots s_{f^2(j_{k-2})} s_{f^2(j_{k-1})} s_{f^3(j_{k-1})} \cdots s_{f^{2k-1}(j_1)} \\
= s_{j_1} \cdots s_{j_k} s_{f(j_k)}\cdots s_{f^{2k-1}(j_1)},
\end{gather*}
as claimed.

Next, let us show that if $v,w \in \mathcal{L}(f)$ cover a common element, then they have a join $v \vee w \in \mathcal{L}(f)$. Because $\mathcal{L}(f)$ is upho, we may assume that the common element they cover is $\hat{0}$, i.e., we may assume that $v$ and $w$ are atoms. And of course we may assume $v\neq w$. So $v=s_i$ and $w=s_j$ for some $1 \leq i, j \leq n$ with $i \neq j$. But then it is clear that their join is $s_is_{f(i)}=s_js_{f(j)}$. Indeed, let $u \in S^*$ be a word which as an element of $M(f)$ is greater than $s_i$ and $s_j$. Since $u$ cannot start with both $s_i$ and~$s_j$, without loss of generality, assume that it does not start with $s_i$. Because~$s_i \leq u$, we must be able to convert the first letter in $u$ to $s_i$ using relations in $M(f)$. At some point when we do this conversion, our word must start with $s_is_{f(i)}$, so indeed we have $s_is_{f(i)} \leq u$, as desired. 

Thus, $\mathcal{L}(f)$ satisfies the conditions of \cref{lem:bez}, and so $\mathcal{L}(f)$ is a lattice. That its core is $M_n$ is clear, since $s_1s_{f(1)}=\cdots=s_ns_{f(n)}$ is the join of the atoms.
\end{proof}

To finish the proof of our main result in this section, we need to think about when different functions $f\colon [n]\to [n]$ yield different upho lattices $\mathcal{L}(f)$ in \cref{thm:rank2_monoids}. As we hinted at in the proof of \cref{thm:dv_flip_lattices}, what matters is the structure of the fibers of~$f$. More precisely, we have the following.

\begin{lemma} \label{lem:partitions}
Let $n\geq 2$. For any partition $\lambda = (\lambda_1,\lambda_2,\ldots,\lambda_\ell)\vdash n$ of the integer~$n$, define a function $f_{\lambda} \colon [n] \to [n]$ by letting
\begin{align*}
    f_{\lambda}(1)= f_{\lambda}(2)= \cdots= f_{\lambda}(\lambda_1) &\coloneqq \lambda_1,\\
    f_{\lambda}(\lambda_1+1) = f_{\lambda}(\lambda_1+2) = \cdots = f_{\lambda}(\lambda_1+\lambda_2) &\coloneqq \lambda_1+\lambda_2,\\
    f_{\lambda}(\lambda_1+\lambda_2+1) = f_{\lambda}(\lambda_1+\lambda_2+2) = \cdots = f_{\lambda}(\lambda_1+\lambda_2+\lambda_3) &\coloneqq \lambda_1+\lambda_2+\lambda_3, \\
    &\;\; \vdots \\
    f(\lambda_1+\lambda_2+\cdots+\lambda_{{\ell-1}}+1)= \cdots = f(n) &\coloneqq n.
\end{align*}
(So $f_{\lambda}$ is idempotent, and $\lambda$ is the partition of the fiber sizes.) Then, with the notation of \cref{thm:rank2_monoids}, the lattices $\mathcal{L}(f_{\lambda})$ and $\mathcal{L}(f_{\nu})$ for partitions $\lambda,\nu\vdash n$ are isomorphic if and only if $\lambda=\nu$.
\end{lemma}

\begin{proof}
Let $\lambda = (\lambda_1,\ldots,\lambda_\ell) \vdash n$ be a partition, and consider the upho lattice $\mathcal{L}(f_{\lambda})$. The elements $s_{\lambda_1}^3, \, s_{\lambda_1+\lambda_2}^3, \, s_{\lambda_1+\lambda_2+\lambda_3}^3, \, \ldots, \, s_n^3 \in \mathcal{L}(f_{\lambda})$ of rank three cover
\[\lambda_1(n-1)+1, \, \lambda_2(n-1)+1, \, \lambda_3(n-1)+1, \, \ldots, \, \lambda_{\ell}(n-1)+1\]
elements of rank two. (For example, $s_{\lambda_1}^3$ covers elements of the form $s_is_j$ where $1 \leq i \leq \lambda_1$ and $1 \leq j \leq n$, of which there are $\lambda_1(n-1)+1=\lambda_1n-(\lambda_1-1)$ because we have the equalities $s_1s_{\lambda_1}=s_2s_{\lambda_2}=\cdots=s_{\lambda_1}s_{\lambda_1}$.) Moreover, these are all the elements of rank three which cover more than one element. Hence, for two different partitions $\lambda,\nu\vdash n$, the multisets of numbers of lower covers for rank-three elements in $\mathcal{L}(f_{\lambda})$ and $\mathcal{L}(f_{\nu})$ will be different, and so $\mathcal{L}(f_{\lambda})$ and $\mathcal{L}(f_{\nu})$ must be non-isomorphic.
\end{proof}

\Cref{thm:rank2_monoids} and \cref{lem:partitions} together establish our main result of this section.

\begin{thm}
For any $n \geq 2$, we have $\kappa(M_n) \geq p(n)$, the number of integer partitions of~$n$.
\end{thm}

\begin{proof}
This follows immediately from \cref{thm:rank2_monoids} and \cref{lem:partitions}.
\end{proof}

\begin{remark}
After seeing the first version of this paper, David Speyer explained to us how our methods can be extended to show that $\kappa(M_n) \geq F_{n-3}$ for all $n \geq 3$, where~$F_n$ is the $n$th Fibonacci number. (Recall that the \dfn{Fibonacci numbers} are defined recursively by $F_0=0$, $F_1=1$, and $F_n = F_{n-1}+F_{n-2}$ for $n\geq 2$. Unlike the partition numbers, which grow super-polynomially but sub-exponentially, the Fibonacci numbers grow exponentially.) We thank him for allowing us to include a sketch of his argument here. Given a composition $\alpha=(\alpha_1,\ldots,\alpha_k)$ of $n-2$ into parts of size at least $2$, define the function $f_{\alpha}\colon [n] \to [n]$ as follows:
\begin{itemize}
\item $f_{\alpha}(1)=1$;
\item $f_{\alpha}(i)=i-1$ for $2 \leq i \leq 2k+2$;
\item if $2(k-j+2)+(\alpha_1+\alpha_2+\cdots+\alpha_{j-1})< i \leq 2(k-j+1)+(\alpha_1+\alpha_2+\cdots+\alpha_{j})$ for $1 \leq j \leq k$, then $f_{\alpha}(i)=2j+1$.
\end{itemize}
Then one can show that for two such compositions $\alpha$ and $\beta$ of $n-2$, the upho lattices $\mathcal{L}(f_{\alpha})$ and $\mathcal{L}(f_{\beta})$ are isomorphic if and only if $\alpha=\beta$. And for $n \geq 3$, there are $F_{n-3}$ such compositions~$\alpha$ of~$n-2$.
\end{remark}

To conclude this section, we note that while we have seen $\kappa(M_2)=2$, it remains possible that $\kappa(M_n)$ is infinite for some $n \geq 3$, even possibly for $n=3$. Indeed, the~$M_n$ have large automorphism groups, so \cref{thm:finite} does not apply. Of course, in light of \cref{cor:monoid_lat_finiteness}, if $\kappa(M_n)$ were infinite for some $n \geq 3$, it would have to be because there are non-colorable upho lattices $\mathcal{L}$ with core $M_n$.

\section{Listing all the ways a finite lattice can be a core} \label{sec:algorithm}

By now we see that the following question is central to understanding all the ways a given finite lattice can be realized as a core.

\begin{question} \label{quest:monoid}
Does every upho lattice come from a monoid; i.e., in the language of \cref{sec:color}, is every upho lattice colorable?
\end{question}

We are not sure whether we should expect a positive answer to \cref{quest:monoid}. Certainly, \cref{sec:color} provides some reasons to think the answer might be positive. So suppose for the moment that \cref{quest:monoid} does have a positive answer. Then from \cref{cor:monoid_lat_finiteness} it would immediately follow that $\kappa(L)$ is finite for all finite graded lattices $L$. But it would still not be clear how to list all the upho lattices of which~$L$ is a core. In this section, we speculate about how one could devise an algorithm for listing all the (colorable) upho lattices of which a given finite lattice is a core.

First of all, we note that for $L$ to be a core of an upho lattice, an obvious requirement is that it be a finite graded lattice whose maximum element $\hat{1}$ is the join of its atoms. Hence, we will only consider $L$ of this form.

Let $L$ be a finite graded lattice for which $\hat{1}$ is the join of the atoms. A \dfn{pre-upho coloring}\footnote{Our notion of pre-uhpo coloring of a finite graded lattice is similar to the notion of ``colored semi-upho poset'' considered in~\cite{fu2024upho}.} of $L$ is a function $c$ that maps each cover relation of $L$ to an atom of $L$, such that:
\begin{itemize}
\item $c(\hat{0}\lessdot s) = s$ for every atom $s\in L$;
\item for each $x \in L \setminus\{\hat{0},\hat{1}\}$, letting $y_1,\ldots,y_k$ be the covers of $x$, there is a rank- and color-preserving embedding of the interval $[x,y_1\vee\cdots\vee y_k]$ into~$L$.
\end{itemize}
(By an \dfn{embedding} of $P$ into $Q$ we mean a map $\varphi\colon P \to Q$ which is an isomorphism onto its image. That this isomorphism is \dfn{rank-preserving} means $\rho(\varphi(p))=\rho(p)$ for all $p \in P$; that it is \dfn{color-preserving} means $c(\varphi(x) \lessdot \varphi(y))=c(x \lessdot y)$ for all~$x \lessdot y \in P$.) Note that a rank-preserving embedding of $[x,y_1\vee\cdots\vee y_k]$ into~$L$ must in particular send $x$ to $\hat{0}$ and $y_1,\ldots,y_k$ to atoms of $L$.

The following is a colored version of \cite[Lemma~5.11]{hopkins2024upho1}.

\begin{lemma}
Let $\mathcal{L}$ be an upho lattice with core $L$. Let $c$ be an upho coloring of $\mathcal{L}$. Then $c$ restricts to a pre-upho coloring of $L$.
\end{lemma}

\begin{proof}
Let $x \in L \setminus\{\hat{0},\hat{1}\}$ and let $y_1,\ldots,y_k\in L$ be the elements covering $x$. Let $\varphi_x\colon V_x \to \mathcal{L}$ be an isomorphism verifying that the coloring $c$ of $\mathcal{L}$ is upho. Then $\varphi_x$ restricts to a rank- and color-preserving embedding of $[x,y_1\vee\cdots\vee y_k] \subseteq L$ into~$L$. Indeed, the only thing that needs to be checked is that $\varphi_x([x,y_1\vee\cdots\vee y_k]) \subseteq L$, because that $\varphi_x$ restricts to a rank- and color-preserving embedding of $[x,y_1\vee\cdots\vee y_k]$ into $\mathcal{L}$ is clear. To see that $\varphi_x([x,y_1\vee\cdots\vee y_k]) \subseteq L$, first note that since the core~$L$ is an interval in~$\mathcal{L}$, it is a sublattice of~$\mathcal{L}$. Thus, the join $y_1\vee\cdots\vee y_k$ is the same whether considered in $L$ or~$\mathcal{L}$. Since $y_1\vee\cdots\vee y_k \in \mathcal{L}$ is less than the join of \emph{all} the elements in $\mathcal{L}$ covering $x$ (whose image under $\varphi_x$ is the maximum of the core~$L$), it must be that $\varphi_x(y_1\vee\cdots\vee y_k)$, and hence all of~$\varphi_x([x,y_1\vee\cdots\vee y_k])$, belongs to~$L$.
\end{proof}

So in order for $L$ to be the core of some colored upho lattice, it must admit a pre-upho coloring. What about the converse? Does a pre-uhpo coloring of $L$ give us a colored upho lattice of which $L$ is the core? We speculate about this in the following conjecture.

\begin{conj} \label{conj:precoloring_monoid}
Let $L$ be a finite graded lattice for which $\hat{1}$ is the join of the atoms. Let $c$ be a pre-upho coloring of $L$. Define the monoid $M$ by
\[ M = \left\langle s_1,\ldots,s_r\mid \parbox{3.75in}{\begin{center} $c(\hat{0}=x_0\lessdot x_1)c(x_1\lessdot x_2)\cdots c(x_{k-1}\lessdot x_k = x_1 \vee y_1) =$ \\ $c(\hat{0}=y_0\lessdot y_1)c(y_1\lessdot y_2)\cdots c(y_{k-1}\lessdot y_k = x_1 \vee y_1)$ \end{center} }\right\rangle\]
where the generators are the atoms $s_1,\ldots,s_r$ of $L$, and the relations correspond to all pairs of saturated chains $\hat{0}=x_0\lessdot x_1\lessdot \cdots \lessdot x_k=x_1 \vee y_1$, $\hat{0}=y_0\lessdot y_1\lessdot \cdots \lessdot y_k=x_1 \vee y_1$ from $\hat{0}$ to the join $x_1 \vee y_1$ of two atoms $x_1,y_1$ of $L$. Then $M$ is left-cancellative, and every pair of elements in $M$ has a greatest common left divisor. Hence $\mathcal{L} \coloneqq (M,\leq_L)$ is an upho meet semilattice, and there is a rank-preserving embedding of $L$ into $\mathcal{L}$.
\end{conj}

\begin{remark}
\Cref{thm:rank2_monoids} shows that \cref{conj:precoloring_monoid} is true in the simplest nontrivial case where $L=M_n$ for some $n \geq 2$.  
\end{remark}

Since verifying that a coloring of a finite lattice $L$ is pre-upho is clearly a finite check, \cref{conj:precoloring_monoid}, if correct, would \emph{almost} give us an algorithm for listing all the colorable upho lattices $\mathcal{L}$ of which a given finite lattice $L$ is a core. However, there are a few deficiencies in this conjecture, as we now explain. 

\begin{figure}
\begin{center}
    \begin{tikzpicture}
    \node[circle, inner sep=1pt, fill=black] (1) at (0,0) {};
    \node[circle, inner sep=1pt, fill=black, label=180:{\color{red}$a$}] (2) at (-1.5,1) {};
    \node[circle, inner sep=1pt, fill=black, label=180:{\color{blue}$b$}] (3) at (0,1) {};
    \node[circle, inner sep=1pt, fill=black, label=0:{\color{green!80!black}$c$}] (4) at (1.5,1) {};
    \node[circle, inner sep=1pt, fill=black] (5) at (-0.75,2) {};
    \node[circle, inner sep=1pt, fill=black] (6) at (1.5,2) {};
    \node[circle, inner sep=1pt, fill=black] (7) at (0,3) {};
    \draw[thick,red] (1)--(2)--(5)--(3)--(5)--(7)--(6)--(4);
    \draw[thick,blue] (1)--(3);
    \draw[thick,green!80!black] (1)--(4);
    \end{tikzpicture} 
    \qquad \vrule \qquad 
    \begin{tikzpicture}
    \node[circle, inner sep=1pt, fill=black] (1) at (0,0) {};
    \node[circle, inner sep=1pt, fill=black, label=180:{\color{red}$a$}] (2) at (-1.5,1) {};
    \node[circle, inner sep=1pt, fill=black, label=180:{\color{blue}$b$}] (3) at (0,1) {};
    \node[circle, inner sep=1pt, fill=black, label=0:{\color{green!80!black}$c$}] (4) at (1.5,1) {};
    \node[circle, inner sep=1pt, fill=black] (5) at (-0.75,2) {};
    \node[circle, inner sep=1pt, fill=black] (6) at (1.5,2) {};
    \node[circle, inner sep=1pt, fill=black] (7) at (0,3) {};
    \node[circle, inner sep=1pt, fill=black] (8) at (-2,2) {};
    \node[circle, inner sep=1pt, fill=black] (9) at (0,2) {};
    \node[circle, inner sep=1pt, fill=black] (10) at (0.75,2) {};
    \draw[thick,red] (1)--(2)--(5)--(3)--(5)--(7)--(6)--(4);
    \draw[thick,blue] (1)--(3);
    \draw[thick,green!80!black] (1)--(4);
    \draw[thick,blue] (2)--(8);
    \draw[thick,red] (8)--(7);
    \draw[thick,blue] (3)--(9);
    \draw[thick,red] (9)--(7);
    \draw[thick,blue] (4)--(10);
    \draw[thick,red] (10)--(7);
    \end{tikzpicture}
\end{center}
    \caption{Inputting the coloring of the lattice $L$ on the left into \cref{conj:precoloring_monoid} yields an upho lattice $\mathcal{L}$ whose core, depicted on the right, is bigger than $L$. Here $M=\langle a,b,c\mid aa=ba, aaa=caa\rangle$.} \label{fig:blah1}
\end{figure}
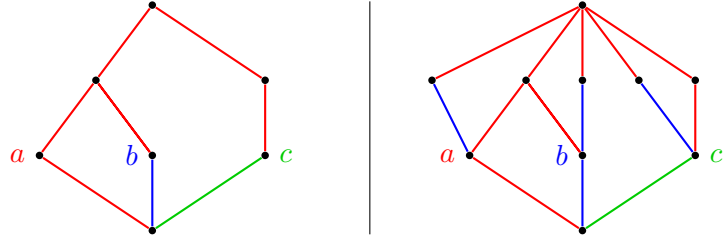

First of all, the conjecture only guarantees that there is a rank-preserving embedding of $L$ into $\mathcal{L}$. This means that $L$ sits inside the core of $\mathcal{L}$, but the core could potentially be bigger than just $L$. Indeed, \cref{fig:blah1} shows an example where the core of the output~$\mathcal{L}$ is strictly bigger than the input~$L$. But this is not such a serious problem for our desired algorithm, as we can check that the core of any output $\mathcal{L}$ is really exactly~$L$ just by looking at a finite portion of $\mathcal{L}$.

Another issue with \cref{conj:precoloring_monoid} is that it only guarantees that $\mathcal{L}$ is a meet semilattice: it may fail to be a lattice because pairs of elements may fail to have upper bounds. Indeed, \cref{fig:blah2} shows an example where the output $\mathcal{L}$ is not a lattice. This is a more serious problem, because checking that upper bounds of all pairs of elements exist is a priori an infinite check.

\begin{figure}
\begin{center}
    \begin{tikzpicture}
    \node[circle, inner sep=1pt, fill=black] (1) at (0,0) {};
    \node[circle, inner sep=1pt, fill=black, label=180:{\color{red}$a$}] (2) at (-1,1) {};
    \node[circle, inner sep=1pt, fill=black, label=0:{\color{blue}$b$}] (3) at (1,1) {};
    \node[circle, inner sep=1pt, fill=black] (4) at (-1,2) {};
    \node[circle, inner sep=1pt, fill=black] (5) at (1,2) {};
    \node[circle, inner sep=1pt, fill=black] (6) at (0,3) {};
    \draw[thick,red] (1)--(2);
    \draw[thick,blue] (1)--(3);
    \draw[thick,blue] (2)--(4)--(6);
    \draw[thick,red] (3)--(5)--(6);
    \end{tikzpicture}
    \qquad \vrule \qquad
    \begin{tikzpicture}[scale=.4]
    \node[circle, inner sep=1pt, fill=black] (1) at (0,0) {};
    \node[circle, inner sep=1pt, fill=black, label=180:{\color{red}$a$}] (2) at (-2,2) {};
    \node[circle, inner sep=1pt, fill=black, label=0:{\color{blue}$b$}] (3) at (2,2) {};
    \node[circle, inner sep=1pt, fill=black] (4) at (-5,4) {};
    \node[circle, inner sep=1pt, fill=black] (5) at (-2,4) {};
    \node[circle, inner sep=1pt, fill=black] (6) at (2,4) {};
    \node[circle, inner sep=1pt, fill=black] (7) at (5,4) {};
    \node[circle, inner sep=1pt, fill=black] (8) at (-8,6) {};
    \node[circle, inner sep=1pt, fill=black] (9) at (-5,6) {};
    \node[circle, inner sep=1pt, fill=black] (10) at (-2,6) {};
    \node[circle, inner sep=1pt, fill=black] (11) at (0,6) {};
    \node[circle, inner sep=1pt, fill=black] (12) at (2,6) {};
    \node[circle, inner sep=1pt, fill=black] (13) at (5,6) {};
    \node[circle, inner sep=1pt, fill=black] (14) at (8,6) {};
    \node[circle, inner sep=1pt, fill=black] (15) at (-11,8) {};
    \node[circle, inner sep=1pt, fill=black] (16) at (-9,8) {};
    \node[circle, inner sep=1pt, fill=black] (17) at (-7,8) {};
    \node[circle, inner sep=1pt, fill=black] (18) at (-5,8) {};
    \node[circle, inner sep=1pt, fill=black] (19) at (-3,8) {};
    \node[circle, inner sep=1pt, fill=black] (20) at (-1,8) {};
    \node[circle, inner sep=1pt, fill=black] (21) at (1,8) {};
    \node[circle, inner sep=1pt, fill=black] (22) at (3,8) {};
    \node[circle, inner sep=1pt, fill=black] (23) at (5,8) {};
    \node[circle, inner sep=1pt, fill=black] (24) at (7,8) {};
    \node[circle, inner sep=1pt, fill=black] (25) at (9,8) {};
    \node[circle, inner sep=1pt, fill=black] (26) at (11,8) {};
    \draw[thick,red] (1)--(2)--(4)--(8)--(15);
    \draw[thick,blue] (1)--(3)--(7)--(14)--(26);
    \draw[thick,blue] (2)--(5)--(11);
    \draw[thick,red] (3)--(6)--(11);
    \draw[thick,blue] (4)--(9)--(18);
    \draw[thick,red] (5)--(10)--(18);
    \draw[thick,blue] (6)--(12)--(23);
    \draw[thick,red] (7)--(13)--(23);
    \draw[thick,blue] (8)--(16);
    \draw[thick,red] (9)--(17);
    \draw[thick,blue] (10)--(19);
    \draw[thick,red] (11)--(20);
    \draw[thick,blue] (11)--(21);
    \draw[thick,red] (12)--(22);
    \draw[thick,blue] (13)--(24);
    \draw[thick,red] (14)--(25);
\end{tikzpicture}
\end{center}
    \caption{Inputting the coloring of the lattice $L$ on the left into \cref{conj:precoloring_monoid} yields the upho meet semilattice $\mathcal{L}$ on the right, which is not a lattice (cf.~\cite[Figure~1]{stanley2021rational}). Here $M=\langle a,b \mid abb=baa \rangle$.} \label{fig:blah2}
\end{figure}
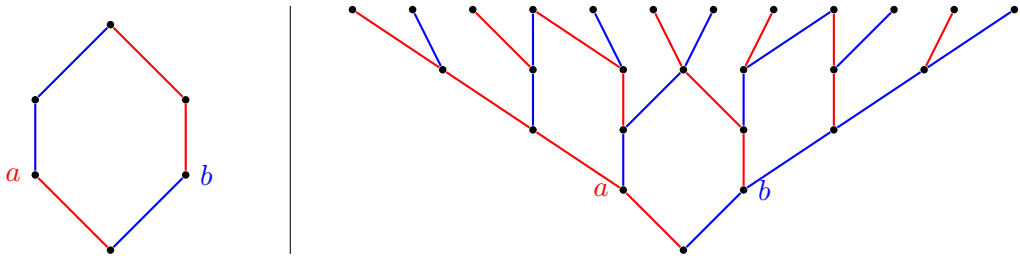

Finally, different monoids $M$ may lead to isomorphic upho lattices $\mathcal{L}$, so if we wanted our algorithm to list each upho lattice only once, we would have to check for isomorphisms, which again is a priori an infinite check. 

\begin{remark} \label{rem:semi}
Since the output of \cref{conj:precoloring_monoid} is an infinite meet semilattice~$\mathcal{L}$, one might wonder whether we could modify this conjecture to also allow as input a finite meet semilattice~$L$. But we have reason to think that this is not possible, at least not in any straightforward way. For example, consider the finite graded meet semilattice $L$ which is obtained from $B_3$ by removing its maximum~$\hat{1}$. We depict a coloring of this $L$ in \cref{fig:blah3}. The monoid we associate to this colored meet semilattice is $M = \langle a,b,c \mid aa = ba, bb = cb, ab = cc \rangle$. But $M$ is not left-cancellative! Indeed, we have $caa = cba = bba = baa = aaa = aba = cca$ in~$M$, even though $aa \neq ca$. The issue is that, although there are no ``local'' violations of left-cancellativity in the coloring of $L$, this coloring cannot be extended to a pre-upho coloring of $B_3$. So it seems the lattice property, in particular, the existence of joins, is really doing something in~\cref{conj:precoloring_monoid}.
\end{remark}

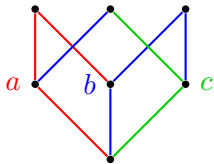
\begin{figure}
\begin{center}
    \begin{tikzpicture}
    \node[circle, inner sep=1pt, fill=black] (1) at (0,0) {};
    \node[circle, inner sep=1pt, fill=black, label=180:{\color{red}$a$}] (2) at (-1,1) {};
    \node[circle, inner sep=1pt, fill=black, label=180:{\color{blue}$b$}] (3) at (0,1) {};
    \node[circle, inner sep=1pt, fill=black, label=0:{\color{green!80!black}$c$}] (4) at (1,1) {};
    \node[circle, inner sep=1pt, fill=black] (5) at (-1,2) {};
    \node[circle, inner sep=1pt, fill=black] (6) at (0,2) {};
    \node[circle, inner sep=1pt, fill=black] (7) at (1,2) {};
    \draw[thick,red] (1)--(2)--(5);
    \draw[thick,red] (3)--(5);
    \draw[thick,blue] (1)--(3)--(7);
    \draw[thick,blue] (2)--(6);
    \draw[thick,blue] (4)--(7);
    \draw[thick,green!80!black] (1)--(4)--(6);
    \end{tikzpicture} 
\end{center}
    \caption{The colored meet semilattice $L$ from~\cref{rem:semi} showing that \cref{conj:precoloring_monoid} does not extend to meet semilattices.} \label{fig:blah3}
\end{figure}

To conclude, we note that the construction in \cref{conj:precoloring_monoid} feels spiritually similar to a known construction of Garside monoids from colored finite lattices discussed in~\cite{mccammond2005introduction}. Hence, techniques from Garside theory~\cite{dehornoy1999gaussian, dehornoy2015foundations} might be useful for proving this conjecture.

\bibliography{upho_lattices_2}{}
\bibliographystyle{abbrv}

\end{document}